\documentclass[a4paper,12pt,twoside]{article}

\usepackage{tensor}
\usepackage[utf8]{inputenc}
\usepackage[UKenglish]{babel}
\usepackage{cite}
\usepackage{amssymb}
\usepackage{amsmath}
\usepackage{mathrsfs}
\usepackage{amsthm}
\usepackage[T1]{fontenc}
\usepackage{accents}
\usepackage[cm]{fullpage}
\usepackage{graphicx}
\usepackage{tensor}
\usepackage{relsize}
\usepackage{appendix}
\usepackage[affil-it]{authblk}
\usepackage{tensor}
\usepackage{verbatim}
\usepackage{comment}
\usepackage{mathtools}
\usepackage{xcolor}
\mathtoolsset{showonlyrefs,showmanualtags}

\allowdisplaybreaks

\RequirePackage{lastpage}
\RequirePackage{latexsym}
\makeatletter
\ifx\@NODS\undefined\RequirePackage{dsfont}\fi
\ifx\@NOAMS\undefined\RequirePackage{amsmath,amsfonts,amssymb,amsthm}\fi
\makeatother

\RequirePackage{geometry}
\RequirePackage{bera}
\RequirePackage[pdftex,pagebackref=false]{hyperref}
\hypersetup{pdfborder=0 0 0}
\hypersetup{pdfstartview={FitH}}

\title{{\Large \bfseries{Approximation of Riemannian\\ measures by Stein's method}}}

\author{James Thompson\footnote{University of Luxembourg, Email: \texttt{james.thompson@uni.lu}}}

\date{\today}

\def\P{\mathbb{P}}
\def\N{\mathbb{N}}
\def\E{\mathbb{E}}
\def\R{\mathbb{R}}
\def\1{\mathbf{1}}
\def\A{\mathcal{A}}
\def\W{\mathcal{W}}
\def\n{\nabla}
\def\<{\langle}
\def\>{\rangle}
\def\{{\lbrace}
\def\}{\rbrace}

\def\Hess{\mathop{\rm Hess}}
\def\Ric{\mathop{\rm Ric}}

\DeclareMathOperator{\Var}{Var}
\DeclareMathOperator{\Cut}{Cut}
\DeclareMathOperator{\Aut}{Aut}
\DeclareMathOperator{\Lip}{Lip}
\DeclareMathOperator{\id}{id}
\DeclareMathOperator{\tr}{tr}
\DeclareMathOperator{\cosech}{cosech}

\newtheorem{theorem}{Theorem}[section]
\newtheorem{lemma}[theorem]{Lemma}
\newtheorem{proposition}[theorem]{Proposition}
\newtheorem{corollary}[theorem]{Corollary}
\newtheorem{definition}[theorem]{Definition}

\begin{document}

\maketitle

\begin{abstract}%
   \noindent%
   {In this article, we present the theoretical basis for an approach to Stein's method for probability distributions on Riemannian manifolds. Using a semigroup representation for the solution to the Stein equation, we use tools from stochastic calculus to estimate the derivatives of the solution, yielding a bound on the Wasserstein distance. We first assume the Bakry-Emery-Ricci tensor is bounded below by a positive constant, after which we deal separately with the case of uniform approximation on a compact manifold. Applications of these results are currently under development and will appear in a subsequent article.}\\[1em]%
   {\footnotesize%
     \par%
     \par
   }
\end{abstract}

\section{Introduction}

Stein's method provides a way to obtain bounds on the distance between two probability distributions with respect to a suitable metric. For a converging sequence of such distributions, it furthermore provides a bound on the rate of convergence. In this article, we develop an approach to Stein's method for probability distributions on \emph{Riemannian manifolds}.
\bigbreak
The history of Stein's method begins in 1972 with the seminal work of Charles Stein on normal approximation \cite{Stein1972}. In this this article, we will not provide a detailed introduction to Stein's method, since many such introductions already exist, and instead refer the reader to the survey \cite{Chatterjee2014}.
\bigbreak
In the article \cite{FangShaoXu2019-1}, together with its corrections \cite{FangShaoXu2019-2}, Fang, Shao and Xu presented an approach to multivariate approximation in the Wasserstein distance, using Stein's method and Bismut's formula. They used Bismut's formula to estimate derivatives of the solution to Stein's equation. We take a similar approach in the general setting of a Riemannian manifold, having to first calculate Bismut-type integration by parts formulas for the gradient, Hessian and \emph{third} derivative of the solution.
\bigbreak
Let us first say a little more about Bismut's formula. For a suitable function $f$ on a complete Riemannian manifold $M$ and a semigroup $P_t$ generated by a diffusion operator of the form
\begin{equation}
\A = \tfrac{1}{2}\Delta + Z
\end{equation}
for $Z$ a smooth vector field, Bismut's formula provides a probabilistic formula for the derivative $\nabla P_tf$ that does not involve the derivative of $f$. Bismut \cite{Bismut1984} proved it on compact manifolds using techniques from Malliavin calculus, after which it was generalized to a full integration by parts formula by Driver \cite{Driver1992}. An elementary approach to derivative formulas, based on martingales, was then developed by Elworthy and Li \cite{ElworthyLi1994}, after which an approach based on local martingales was given by Thalmaier \cite{Thalmaier1997} and Driver and Thalmaier \cite{DriverThalmaier2001}. In this article, we follow the latter approach, calculating our derivative formulas by first identifying appropriate local martingales. While Bismut's formula is well-known, formulas for the Hessian, or second derivative, of the semigroup are not well-known. Nonetheless, various formulas for the Hessian have been given, for example, in \cite{ElworthyLi1994,ArnaudonPlankThalmaier2003,Wang2019} and \cite{Thompson2019}. In this article we present, to the best of the author's knowledge for the first time, a complete set of formulas for the first, second and \emph{third} derivative of the semigroup.
\bigbreak
It is worth noting that there is a relationship between Bismut's formula and Stein's lemma. In particular, Hsu proved in \cite{Hsu2005} that just as the relation
\begin{equation}\label{eq:steinlemma}
\E[f'(X)] = \E[Xf(X)]
\end{equation}
characterizes the standard normal distribution, which roughly speaking is the statement of Stein's lemma, so Driver's integration by parts formula
\begin{equation}\label{eq:ibpdriver}
\E[D_h F(X)] = \E\left[F(X)\int_0^1 \bigg\langle \dot{h}_s + \frac{1}{2}{\Ric}_{U(X)_s} h_s,dW_s\bigg\rangle\right]
\end{equation}
characterizes Brownian motion, at least on a compact manifold, among the set of probability measures on the path space of $M$ for which the coordinate process is a semimartingale. Note that if $M=\R$ with $h_s = s$ and $F(X) = f(X_1)$ then equation \eqref{eq:ibpdriver} becomes equation \eqref{eq:steinlemma}. So Hsu's theorem is an analogue of Stein's lemma for the infinite-dimensional path space setting. Bismut's formula is a special case of equation \eqref{eq:ibpdriver}, namely the one in which $F$ is given by the composition of some $f$ with an evaluation at a fixed time. 
\bigbreak
\textbf{Main results.} Now let us discuss the main results. For a smooth function $\psi$ on a complete Riemannian manifold $M$ with
\begin{equation}
c:= \int_M e^{\psi} d\mu < \infty,
\end{equation}
consider the probability measure $\mu_\psi$ on $M$ given by
\begin{equation}
d\mu_\psi:= \frac{1}{c} e^{\psi} d\mu.
\end{equation}
One example would be the case in which $\psi$ is given by the logarithm of the heat kernel, which would correspond to \emph{heat kernel approximation}, which on $\R^n$ would reduce to \emph{normal approximation}. Another example would be to take $\psi = 0$ on a compact manifold, which would correspond to \emph{uniform approximation}.
\bigbreak
Our main result, which is Theorem \ref{thm:mainthmone}, is then stated as follows: Suppose $K>0$ with $${\Ric}_{\psi}:= \Ric - 2\Hess \psi \geq K$$ where $\Ric$ denotes the Ricci curvature tensor. Suppose
\begin{equation}
\begin{aligned}
&\|R\|_\infty < \infty,\\
&\|\n R\|_\infty < \infty,
\end{aligned}\quad
\begin{aligned}
&\|\nabla {\Ric}^\sharp_\psi + d^\star R -2R(\n \psi)\|_\infty < \infty,\\
&\|\n(\nabla {\Ric}^\sharp_\psi + d^\star R -2R(\n \psi))\|_\infty < \infty\\
\end{aligned}
\end{equation}
where $R$ denotes the full Riemann curvature tensor. Then, denoting by $\W$ the Wasserstein distance, we will prove that there exists a positive constant $C$ such that if $W,W'$ is any pair of identically distributed $M$-valued random variables satisfying
\begin{equation}
\P\{(W,W')\in \Cut\} = 0,
\end{equation}
where $\Cut$ denotes the cut locus of $M$, then
\begin{equation}
\mathcal{W}(\mathcal{L}(W),\mu_\psi) \leq C \left(\frac{1}{\lambda}\E[|\delta|^3(|\log|\delta||\vee 1)] + \E[|R_1|] + \E[|R_2|]\right)
\end{equation}
for all $\lambda >0$, where $|\delta| = d(W,W')$ and where the remainders $R_1$ and $R_2$ are defined, in terms of $\psi$, the Riemannian metric and $\lambda$, in Section \ref{sec:idpairs}.
\bigbreak
When $M=\R^n$, this becomes Fang, Shao and Xu's estimate from \cite{FangShaoXu2019-1}. They showed how such a bound on the Wasserstein distance can indeed be obtained at the price of the $\log|\delta|$. In Theorem \ref{thm:mainthmtwo}, we state a version of the above bound in which the $\log|\delta|$ is dispensed with, but in which the Wasserstein distance is replaced by a $C^2$-distance.
\bigbreak
Note that we assume various bounds on the curvature and Bakry-Emery-Ricci tensors. These assumptions need not be so strong. For example, the uniform boundedness assumptions could be relaxed to allow a certain amount of growth, while in many cases the positive lower bound on $\Ric_{\psi}$ is simply not necessary. The positive lower bound will be used to verify exponential convergence to equilibrium, but as shown for example by Cheng, Thalmaier and Zhang in the recent article \cite{ChengThalmaierZhang2020}, weaker conditions exist. In Section \ref{sec:7}, we show how all assumptions on curvature can be dispensed with for the case $\psi=0$ with $M$ compact, using instead the existence of a spectral gap, yielding Theorems \ref{thm:mainthmonecompact} and \ref{thm:mainthmonecompact}. A weakening of the assumptions in the non-compact case is a topic for future consideration, but for the time being, we will stick to the boundedness assumptions for the sake of simplicity.
\bigbreak
A robust formulation of Stein's method for the manifold setting should yield a central limit theorem for geodesic random walks and estimates on the rates of convergence for these and other sampling algorithms. While this article presents the theoretical basis for this work, the applications themselves are currently under development and will be publicized later.
\bigbreak
It is becoming increasingly important to study probability distributions on manifolds, for a variety of reasons. For example, there is growing interest in the stochastic analysis of manifolds among members of the deep learning community. Over the last decade, deep learning methods based on artificial neural networks have achieved unprecedented performance on a broad range of problems arising in a variety of different contexts, but research has mainly focused on data belonging to Euclidean domains. \emph{Geometric deep learning} seeks to extend these techniques to geometrically structured data, extracting non-linear features, facilitating dimensionality reduction and the use of pattern recognition or classification algorithms. For a survey of this topic see \cite{BBLSV2017}.
\bigbreak
\textbf{Organization of the paper.} This article is targeted primarily at those who are more familiar with Stein's method than with Riemannian geometry. Many of the computational details involving curvature tensors will therefore be left to later sections. Moreover we start in Section \ref{sec:2} with a brief review of Riemannian geometry, including the Riemannian distance function, volume measure and Ricci curvature. Afterwards, in Section \ref{sec:3}, we look at Stein's equation and show how its solution can be written in terms of a semigroup. In Sections \ref{sec:4} and \ref{sec:idpairs} we show how Taylor expansion can be used to analyse the difference between two identically distributed manifold-valued random variables. We then present the main results in Section \ref{sec:6}. In Section \ref{sec:7} we consider the special case of a compact manifold with $\psi=0$, where the assumption of positive Ricci curvature can be dropped due to the existence of a spectral gap. In Section \ref{sec:8} we provide a selection of examples of Riemannian measures which satisfy the criterion $\Ric_{\psi}>K>0$. The remainder of the paper consists of the calculations needed to obtain the derivative estimates used in the proofs of the main theorems in Section \ref{sec:6}. This starts in Sections \ref{sec:comm} and \ref{sec:10} with the required commutation relations, which are then used to obtain the derivative formulas in Section \ref{sec:11}, which are finally used to derive the necessary derivative estimates in Section \ref{sec:derests}.

\section{Riemannian measure}\label{sec:2}

A topological space $M$ is called \emph{locally Euclidean} if there exists a non-negative integer $n$ such that every point in $M$ has a neighbourhood that is homeomorphic to the Euclidean space $\R^n$. The number $n$ is referred to as the \emph{dimension} of $M$.

\begin{definition}
A locally Euclidean second-countable Hausdorff space is called a \emph{topological manifold}.
\end{definition}

While topological manifolds inherit many of the local properties of Euclidean space, such as local compactness and local connectedness, the Hausdorff property is non-local and must therefore be assumed.

\begin{definition}
A topological manifold equipped with a smooth structure is called a \emph{smooth manifold}.
\end{definition}

A \emph{smooth structure} gives meaning to smooth functions, tangent spaces and vectors fields. It allows for calculus on the manifold to be performed unambiguously. For a smooth manifold $M$, second-countability and the Hausdorff property are, according to the Whitney embedding theorem, precisely the conditions required to ensure the existence of a smooth map, between $M$ and some finite-dimensional Euclidean space, under which $M$ is diffeomorphic to its image. Such a map is called an \emph{embedding} and it can be shown that the smallest Euclidean space into which all $n$-dimensional smooth manifolds can be embedded is $\R^{2n}$.
\bigbreak
For each point $x \in M$, the \emph{tangent space at $x$}, denoted $T_xM$, is the $n$-dimensional vector space consisting of the equivalence classes of velocities at $x$ of all smooth curves passing through $x$, with two velocities said to be equivalent if their curves when composed with a local homeomorphism have velocities equal in $\R^n$. The disjoint union of the tangent spaces is called the \emph{tangent bundle} and denoted $TM$. A smooth assignment of tangent vectors to each tangent space is called a \emph{vector field}. A smooth assignment of cotangent vectors to each cotangent space $(T_xM)^\star$ is called a \emph{differential $1$-form}. Given a smooth function $f$, the \emph{differential} of $f$ is the unique $1$-form $df$ such that for all vector fields $V$ and points $x \in M$ the composition $(df)_x(V(x))$ coincides with the directional derivative $V(f)(x)$ at $x$ of $f$ in the direction $V$.
\bigbreak
A \emph{Riemannian metric} is a smooth assignment of inner products to each tangent space. Denoting the inner products by $\<\cdot,\cdot\>$, this means that if $V_1$ and $V_2$ are vector fields then the map $x \mapsto \< V_1(x),V_2(x)\>_x$ is smooth.

\begin{definition}
A smooth manifold equipped with a Riemannian metric is called a \emph{Riemannian manifold}.
\end{definition}

By a theorem of Nash, a Riemannian manifold can always be isometrically embedded into a Euclidean space but, as with the Whitney embeddings, there is not a canonical way of doing so. Riemannian manifolds support a number of intrinsically defined objects, familiar from $\R^n$, the most important of which we will now briefly describe.
\bigbreak
\textbf{Riemannian distance.} For a Riemannian manifold $M$, the \emph{Riemannian distance} $d$, with respect to which $M$ is a metric space, is defined for $x,y \in M$ by
\begin{equation}
d(x,y) := \inf_{\gamma} \int_0^1 | \dot{\gamma}(t) |_{\gamma(t)} dt
\end{equation}
where the infimum is over all piecewise smooth curves $\gamma: [0,1]\rightarrow M$ with $\gamma(0) = x$ and $\gamma(1)=y$. We will assume that $M$ equipped with this distance is a complete metric space. If we denote by $\mathcal{P}(M)$ collection of all probability measures $\nu$ on $M$ such that
\begin{equation}
\int_M d(x,y) d \nu (y) < + \infty
\end{equation}
for some $x \in M$ then, for two such measures $\nu_1$ and $\nu_2$, the \emph{Wasserstein distance} between $\nu_1$ and $\nu_2$ is given by
\begin{equation}
\W (\nu_1,\nu_2) = \sup\lbrace| \nu_1(h)-\nu_2(h)|:h \in \Lip_1(M)\rbrace
\end{equation}
where $\Lip_1(M)$ denotes the set of functions with $|h(x)-h(y)| \leq d(x,y)$ for any $x,y \in M$.
\bigbreak
\textbf{Volume measure.} Viewing $(M,d)$ as a metric space, the \emph{Riemannian volume measure} $\mu$ is then the normalized $n$-dimensional Hausdorff measure restricted to the $\sigma$-algebra of Borel sets. It is normalized in the sense that it agrees with the Lebesgue measure on $\R^n$.
\bigbreak
\textbf{Laplace-Beltrami operator.} Given the measure $\mu$, we denote by $\Delta$ the unique second-order partial differential operator with the property that
\begin{equation}\label{eq:laplacian}
\int_M f \Delta g\, d\mu  = - \int_M \langle df,dg\rangle\, d\mu
\end{equation}
for all compactly supported smooth functions $f$ and $g$, where $\langle \cdot,\cdot \rangle$ denotes the inner product induced on a cotangent space by Riesz isomorphism. It is called the \emph{Laplace-Beltrami operator}.
\bigbreak
\textbf{Brownian motion.} By a \emph{diffusion} we mean a continuous time Markov process with almost surely continuous sample paths. A diffusion on $M$ with differential generator $\frac{1}{2}\Delta$ is called a \emph{Brownian motion}. More generally, suppose $Z$ is a smooth vector field on $M$ and for each $x \in M$ suppose $X(x)$ is a diffusion on $M$ starting at $x \in M$ with generator
\begin{equation}\label{eq:genA}
\A:= \tfrac{1}{2}\Delta +Z
\end{equation}
and explosion time $\zeta(x)$. The random time $\zeta(x)$, which could be infinite, is the first time at which $X(x)$ leaves all compact subsets of $M$. Note that on an arbitrary smooth manifold, any elliptic second-order partial differential operator with smooth coefficients and vanishing zeroth order part induces, via its principal symbol, a Riemannian metric with respect to which it takes precisely the form \eqref{eq:genA} for some vector field $Z$. We define the associated minimal semigroup $\lbrace P_t : t\geq 0\rbrace$ acting on bounded measurable functions $f$ by the formula
\begin{equation}
P_tf(x) := \E[\1_{\lbrace t<\zeta(x)\rbrace} f(X_t(x))]
\end{equation}
which on $(0,\infty) \times M$ is smooth and satisfies the diffusion equation
\begin{equation}
(\partial_t -\A)P_tf = 0
\end{equation}
with initial condition $P_0f = f$.
\bigbreak
\textbf{Ricci curvature.} For a Riemannian manifold $M$, of dimension $n$, the Ricci curvature of $M$, denoted $\Ric$, is a tensor field on $M$ which acts as a symmetric bilinear form on each tangent space. Its precise definition is given in Section \ref{sec:comm}. It provides a way to measure the degree to which the Riemannian geometry of $M$ differs from the Euclidean geometry of $\R^n$. For example, in normal coordinates centred at a point $p\in M$, the volume measure $\mu$ satisfies at $p$ the Taylor series expansion
\begin{equation}
d\mu = \left( 1 - \frac{1}{6}{\Ric}_{ij} x^i x^j + O(|x|^3)\right) d x^1 \cdots dx^n
\end{equation}
where ${\Ric}_{ij}$ denote the components, in these coordinates, of the Ricci curvature tensor. Ricci curvature further relates to Brownian motion via the Weitzenb\"{o}ck formula, which we will introduce later. For a diffusion with generator $\A$ the relevant object is the \emph{Bakry-Emery-Ricci tensor} defined by
\begin{equation}
{\Ric}_Z := \Ric - 2\n Z.
\end{equation}
For $K \in \R$ we write ${\Ric}_Z \geq K$ if
\begin{equation}
{\Ric}_Z(v,v) \geq K |v|^2
\end{equation}
for all tangent vectors $v$, that is, if ${\Ric}_Z$ is bounded below by $K$ in the sense of bilinear forms. It is well known that if ${\Ric}_Z$ is bounded below by a constant then $P_t 1 = 1$, which is to say that for each $x \in M$ the explosion time $\zeta(x)$ is almost surely infinite. Moreover, we have the following theorem:

\begin{theorem}\label{thm:exponentialconv}
The following are equivalent:
\begin{itemize}
\item[(i)] ${\Ric}_Z \geq 2K$;
\item[(ii)] $|\n P_t f| \leq e^{-Kt} P_t|\n f|$ for any $f \in C^1_b(M)$ and $t \geq 0$;
\item[(iii)] $\W(P_t^\ast \nu_1 , P_t^\ast \nu_2) \leq e^{-Kt} \W(\nu_1,\nu_2)$ for any probability measures $\nu_1,\nu_2$ and $t \geq 0$.
\end{itemize}
\end{theorem}

\begin{proof}
See, for example, the proof of \cite[Theorem~2.3.3]{Wang2014}.
\end{proof}

\section{Stein equation}\label{sec:3}

Now suppose $\psi$ is a smooth function on $M$ such that
\begin{equation}
c:= \int_M e^{\psi} d\mu < \infty.
\end{equation}
For such $\psi$ consider the probability measure $\mu_\psi$ on $M$ defined by
\begin{equation}
d\mu_\psi:= \frac{1}{c} e^{\psi} d\mu.
\end{equation}
By \eqref{eq:laplacian} it follows that the measure $\mu_\psi$ is invariant for the symmetric operator
\begin{equation}
\A = \frac{1}{2}\Delta + \n \psi
\end{equation}
which fits into the framework of the previous section by setting $Z = \n \psi$. Note here that $\n \psi$ denotes the gradient of $\psi$. Suppose also that $\mu_\psi$ has finite first moments. That is, suppose
\begin{equation}
\int_M d(x,y) d\mu_\psi (y) < \infty
\end{equation}
for some $x \in M$ and therefore, by the triangle inequality, for every $x \in M$. Then
\begin{align}
\W(\delta_x,\mu_\psi)=\,& \sup_{\Lip_1(M)} |h(x) - \mu_\psi(h)|\\
\leq \,& \sup_{\Lip_1(M)} \int_M |h(x) - h(y)| d\mu_\psi(y)\\
\leq\,& \int_M d(x,y) d\mu_\psi(y)
\end{align}
which is finite. Consequently, if we set
\begin{equation}
{\Ric}_\psi := \Ric - 2 \Hess \psi
\end{equation}
and assume ${\Ric}_\psi$ is bounded below by a positive constant then, by Theorem \ref{thm:exponentialconv}, it follows that the semigroup $P_t$ converges exponentially fast to equilibrium with $\mu_\psi$ the unique invariant ergodic measure. For example, if $\Ric$ is bounded below by a positive constant then the normalized Riemannian measure $\mu$ is the unique invariant probability measure for Brownian motion. Note that in this case, according to the Bonnet-Myers theorem, $M$ must be compact and therefore of finite volume. More generally, if $K>0$ with
\begin{equation}
{\Ric}_\psi \geq 2K
\end{equation}
then Theorem \ref{thm:exponentialconv} implies
\begin{equation}
\sup_{\Lip_1(M)} |P_th(x) - \mu_\psi(h)|= \W(P^\ast_t \delta_x,\mu_\psi)\leq e^{-Kt} \W(\delta_x,\mu_\psi)
\end{equation}
and so we have
\begin{equation}
\sup_{\Lip_1(M)} \bigg| \int_0^\infty P_th(x) - \mu_\psi(h) ds\bigg| < \infty.
\end{equation}
Consequently, if given a function $h\in \Lip_1(M)$ we suppose $f$ is a solution to the \emph{Stein equation}
\begin{equation}\label{eq:stein}
\mathcal{A}f = h - \mu_\psi(h)
\end{equation}
then $f$ is given by the formula
\begin{equation}\label{eq:poissoln}
f(x) = - \int_0^\infty P_th(x) - \mu_\psi(h) dt
\end{equation}
for all $x \in M$.

\section{Taylor expansion}\label{sec:4}

The \emph{cut locus} of $M$, denoted $\Cut$, is the the closure in $M \times M$ of the set of all pairs of points at which the squared distance function $d^2$ fails to be differentiable. It is a set of $\mu$-measure zero. For pairs of points outside the cut locus, there exists a unique length-minimizing geodesic segment connecting the two points. If $w,w' \in M$ are points outside the cut locus then we denote by $\gamma:=\gamma_{w,w'} :[0,1]\rightarrow M$ the unique minimizing geodesic segment with $\gamma(0) = w$ and $\gamma(1) = w'$. The geodesic $\gamma$ has initial velocity
\begin{equation}
\delta:= \dot{\gamma}(0) \in T_w M
\end{equation}
and speed $|\delta| = d(w,w')$. For an arbitrary function $f$ that is sufficiently continuously differentiable, defining $\phi:[0,1] \rightarrow \R$ by $\phi(t):= f(\gamma(t))$ we see, by Taylor's theorem with the remainder in integral form, that
\begin{align}
\phi(t) =\,& \phi(0) + \phi'(0)t + \int_0^t \phi''(s)(t-s)ds\\
=\,& \phi(0) + \phi'(0)t + \frac{1}{2}\phi''(0) t^2 + \frac{1}{2}\int_0^t \phi'''(s)(t-s)^2ds
\end{align}
for all $t \in [0,1]$. We will make use of both of the above equations. Since $\gamma$ is a geodesic it follows that
\begin{align}
\phi'(s) =\,& \frac{d}{ds} f(\gamma(s)) = (df)_{\gamma(s)}(\dot{\gamma}(s))\\
\phi''(s) =\,& \frac{d}{ds} (df)_{\gamma(s)}(\dot{\gamma}(s)) = (\n df)_{\gamma(s)} (\dot{\gamma}(s),\dot{\gamma}(s))\\
\phi'''(s) =\,& \frac{d}{ds} (\n df)_{\gamma(s)} (\dot{\gamma}(s),\dot{\gamma}(s)) = (\n \n df)_{\gamma(s)} (\dot{\gamma}(s),\dot{\gamma}(s),\dot{\gamma}(s))
\end{align}
and therefore, since $\delta= \dot{\gamma}(0)$, setting $t = 1$ we find that if $f\in C^3(M)$ then
\begin{align}
&f(w')-f(w)\\
=\,& (df)_w(\delta) + \frac{1}{2}(\n df)_{w} (\delta, \delta) + \frac{1}{2}\int_0^1 (\n \n df)_{\gamma(s)} (\dot{\gamma}(s),\dot{\gamma}(s),\dot{\gamma}(s))(1-s)^2 ds.
\end{align}
Moreover, denoting by $//$ the parallel transport along $\gamma$, since
\begin{equation}
(\n df)_{\gamma(s)} (\dot{\gamma}(s),\dot{\gamma}(s)) = (\n df)_{\gamma(s)} (//_s \delta,//_s \delta)= (//_s^{\otimes 2})^{-1}(\n df)_{\gamma(s)}(\delta, \delta)
\end{equation}
we find if $f \in C^2(M)$ then
\begin{align}
&f(w')-f(w)\\
=\,& (df)_w(\delta) + \int_0^1 (\n df)_{\gamma(s)} (\dot{\gamma}(s),\dot{\gamma}(s))(1-s)ds \\
=\,& (df)_w(\delta) + \frac{1}{2}(\n df)_{w} (\delta, \delta) +  \int_0^1 ((//_s^{\otimes 2})^{-1}(\n df)_{\gamma(s)}- (\n df)_{w}) (\delta, \delta)(1-s)ds.
\end{align}
Recall that this length-minimizing geodesic $\gamma = \gamma_{w,w'}$ is uniquely determined so long as $w$ and $w'$ are outside the cut locus. 

\section{Identical pairs}\label{sec:idpairs}

Now suppose $W$ is an $M$-valued random variable. Suppose $W'$ is another $M$-valued random variable, defined on the same probability space, with the same distribution as $W$ and
\begin{equation}
\P\{(W,W')\in \Cut\} = 0. \tag{A0}
\end{equation}
Denote by $\delta$ the initial velocity of the unique minimizing geodesic segment connecting $W$ and $W'$, as in the previous section. Fix $\lambda >0$ and define the $T_W M$-valued random variable $R_1$ according to the relation
\begin{equation}
\E[\delta | W] = \lambda (R_1 + (\nabla \psi)(W)) \tag{A1}
\end{equation}
and define the $T_WM \otimes T_WM$-valued random variable $R_2$ according to the relation
\begin{equation}
\E[\delta^{\otimes 2} | W] = \lambda \left(2 R_2 + \sum_{i=1}^n e_i(W) \otimes e_i(W)\right) \tag{A2}
\end{equation}
where $\{ e_i(W)\}_{i=1}^n$ denotes an orthonormal basis of the tangent space $T_WM$. For the case in which $M = \R^n$ the sum of these tensor products reduces simply to the identity matrix. By (A1) and (A2) it follows that
\begin{align}
&\E[(df)_W(\delta)] + \tfrac{1}{2} \E[(\n d f)_W(\delta,\delta)]\\[2mm]
=\,& \E[(df)_W(\E[\delta | W])] + \tfrac{1}{2} \E[(\n d f)_W(\E[\delta^{\otimes 2} | W])]\\[2mm]
=\,& \lambda \E[ (df)_W(\n \psi)] + \lambda \E[(df)_W(R_1) ] + \tfrac{\lambda}{2} \E[\tr (\n d f)_W] + \lambda \E[(\n d f)_W(R_2)]\\[2mm]
=\,& \lambda \E[ (\tfrac{1}{2}\Delta + \nabla \psi)f(W)] + \lambda \E[(df)_W(R_1)]+ \lambda \E[(\n d f)_W(R_2)]
\end{align}
and consequently
\begin{align}
\lambda \E[ \mathcal{A}f(W)] =\,& \E[(df)_W(\delta)] + \tfrac{1}{2} \E[(\n d f)_W(\delta,\delta)]\\
&-\lambda \E[(df)_W(R_1)]- \lambda\E[(\n d f)_W(R_2)]. \label{eq:steinrear}
\end{align}
This is the formula on which our estimates will be based.

\section{Main results}\label{sec:6}

In this section we present the main results. We consider the case in which $\Ric_\psi \geq 2K$ for some constant $K > 0$, since this allows for an application of Theorem \ref{thm:exponentialconv}. For the Ricci flat case, this is to say that $\psi$ is uniformly logarithmically concave.

\begin{proposition}\label{prop:steinest1}
Suppose $K > 0$ with $\Ric_\psi \geq 2K$ and $h \in C^1_b(M)$ and that $f$ is a solution to the Stein equation \eqref{eq:stein}. Then $\|\n f\|_\infty \leq \frac{1}{K}\|\n h\|_\infty$.
\end{proposition}

\begin{proof}
This follows from Theorem \ref{thm:estone} and formula \eqref{eq:poissoln}.
\end{proof}

\begin{proposition}\label{prop:steinest2}
Suppose $K > 0$ with $\Ric_\psi \geq 2K$ and $h \in C^1_b(M)$. Suppose
\begin{equation}
\|R\|_\infty < \infty ,\quad \|\nabla {\Ric}^\sharp_\psi + d^\star R -2R(\n \psi)\|_\infty < \infty
\end{equation}
and that $f$ is a solution to the Stein equation \eqref{eq:stein}. Then there exists a positive constant $c_1$ such that $\|\n d f\|_\infty \leq  c_1 \|\n h\|_\infty$.
\end{proposition}

\begin{proof}
This follows from Theorem \ref{thm:esttwo} since that theorem implies
\begin{equation}
|\n d f|(x) \leq \int_0^\infty |\n d P_th|(x) dt \leq C_1 \|\n h\|_\infty  \int_0^\infty \frac{1}{\sqrt{t \wedge 1}}e^{-Kt}dt = c_1 \|\n h\|_\infty
\end{equation}
by formula \eqref{eq:poissoln}, as required.
\end{proof}

\begin{theorem}\label{thm:mainthmone}
Suppose $K>0$ with ${\Ric}_\psi \geq 2K$. Suppose
\begin{equation}
\begin{aligned}
&\|R\|_\infty < \infty,\\
&\|\n R\|_\infty < \infty,
\end{aligned}\quad
\begin{aligned}
&\|\nabla {\Ric}^\sharp_\psi + d^\star R -2R(\n \psi)\|_\infty < \infty,\\
&\|\n(\nabla {\Ric}^\sharp_\psi + d^\star R -2R(\n \psi))\|_\infty < \infty.\\
\end{aligned}
\end{equation}
Then there exists a positive constant $C$ such that for all $\lambda >0$ and all pairs $W,W'$ of identically distributed random variables taking values in $M$ and satisfying (A0), with $R_1$ and $R_2$ defined by (A1) and (A2), we have
\begin{equation}
\mathcal{W}(\mathcal{L}(W),\mu_\psi) \leq C \left(\frac{1}{\lambda}\E[|\delta|^3(|\log|\delta||\vee 1)] + \E[|R_1|] + \E[|R_2|]\right)
\end{equation}
where $|\delta| = d(W,W')$.
\end{theorem}

\begin{proof}
Suppose $f$ solves the Stein equation \eqref{eq:stein} for a function $h \in C^1_b(M)$. Then it follows that $f$ is three times differentiable. We will later pass to Lipschitz functions by approximation. For any pair of points $(w,w') \not\in \Cut$ set
\begin{equation}
f_1 := -\int_0^{|\delta|^2} P_th + \mu_\psi(h) dt, \quad f_2 := -\int_{|\delta|^2}^\infty P_th + \mu_\psi(h) dt
\end{equation}
with $\delta \in T_wM$ defined as above. By formula \eqref{eq:poissoln} and Taylor expansion it follows that $f = f_1 + f_2$ with
\begin{align}
f(w')-f(w) =\,& (df)_w(\delta) + \frac{1}{2}(\n df)_{w} (\delta, \delta) \\
& + \int_0^1 ((//_s^{\otimes 2})^{-1}(\n df_1)_{\gamma(s)}- (\n df_1)_{w}) (\delta, \delta)(1-s)ds\\
& + \frac{1}{2}\int_0^1 (\n \n df_2)_{\gamma(s)} (\dot{\gamma}(s),\dot{\gamma}(s),\dot{\gamma}(s))(1-s)^2 ds.
\end{align}
In particular, since $W$ and $W'$ are identically distributed we have
\begin{align}
0 =\,& \E[(df)_W(\delta)] + \frac{1}{2}\E[(\n df)_{W} (\delta, \delta)] \\
& + \int_0^1 \E[((//_s^{\otimes 2})^{-1}(\n df_1)_{\gamma(s)}- (\n df_1)_{W}) (\delta, \delta)](1-s)ds\\
& + \frac{1}{2}\int_0^1 \E[(\n \n df_2)_{\gamma(s)} (\dot{\gamma}(s),\dot{\gamma}(s),\dot{\gamma}(s))](1-s)^2 ds.
\end{align}
Therefore, by \eqref{eq:steinrear}, we have
\begin{align}
&|\E[h(W)] - \mu_\psi(h)|\\[2mm]
=\,& |\E[ \mathcal{A}f(W)]|\\[2mm]
=\,&\E[|(df)_W(R_1)|]+ \E[|(\n d f)_W(R_2)|]+\tfrac{1}{\lambda}|\E[(df)_W(\delta)] + \tfrac{1}{2} \E[(\n d f)_W(\delta,\delta)]|\\[2mm]
\leq \,&  \|df\|_\infty \E[|R_1|] + \| \n d f\|_\infty \E[|R_2|]\\[2mm]
&+ \frac{1}{\lambda}\int_0^1 |\E[((//_s^{\otimes 2})^{-1}(\n df_1)_{\gamma(s)}- (\n df_1)_{W}) (\delta, \delta)]|(1-s)ds\\
& + \frac{1}{2\lambda}\int_0^1 |\E[(\n \n df_2)_{\gamma(s)} (\dot{\gamma}(s),\dot{\gamma}(s),\dot{\gamma}(s))]|(1-s)^2 ds\\
\leq \,&  \|df\|_\infty \E[|R_1|] + \| \n d f\|_\infty \E[|R_2|] + \frac{1}{\lambda}\E[|\delta|^2 \| \n df_1\|_\infty] + \frac{1}{6\lambda}\E[|\delta|^3 \|\n \n df_2\|_\infty]
\end{align}
and it suffices to bound the derivatives. By Theorems \ref{thm:esttwo} and \ref{thm:estthree} with $Z = \n \psi$, together with formula \eqref{eq:poissoln}, it follows that there exist positive constant $c_2$ and $c_3$ such that
\begin{align}
|\n d f_1|(x) \leq\,& \int_0^{|\delta|^2} |\n d P_th|(x) dt\\
\leq\,& C_1 \|\n h\|_\infty  \int_0^{|\delta|^2} \frac{1}{\sqrt{t \wedge 1}}e^{-Kt}dt= c_2 \|\n h\|_\infty (|\delta| \wedge 1)
\end{align}
and similarly that
\begin{align}
|\n \n df_2|(x) \leq\,& \int_{|\delta|^2}^\infty |\n\n dP_th|(x)dt\\
\leq\,& C_2 \|\n h\|_\infty \int_{|\delta|^2}^\infty \frac{1}{t \wedge 1}e^{-Kt}dt
= c_3 \|\n h\|_\infty (|\log |\delta|| \vee 1) e^{-K (|\delta|\vee 1)^2}.
\end{align}
Consequently, by Propositions \ref{prop:steinest1} and \ref{prop:steinest2}, we have
\begin{align}
&|\E[h(W)] - \mu_\psi(h)|\\[2mm]
\leq\,& \|\n h\|_\infty \bigg( \frac{1}{K} \E[|R_1|] + c_1 \E[|R_2|]\\
& +\frac{c_2}{\lambda} \E[|\delta|^2 (|\delta|\wedge 1)] + \frac{c_3}{6\lambda} \E[|\delta|^3(|\log |\delta|| \vee 1) e^{-K (|\delta|\vee 1)^2}]\bigg)
\end{align}
from which it follows that
\begin{equation}
|\E[h(W)] - \mu_\psi(h)| \leq C \|\n h\|_\infty ( \E[|R_1|] + \E[|R_2|] + \frac{1}{\lambda}\E[|\delta|^3(|\log|\delta||\vee 1)])
\end{equation}
where $|\delta| = d(W,W')$. Now, by \cite[Theorem~1]{AzagraFerreraLopezMesasRangel2007} it follows that for every Lipschitz function $h$ on a Riemannian manifold $M$, for every $\epsilon,r >0$, there exists a function $\tilde{h}\in C^\infty(M)$ such that $| h - \tilde{h}|_\infty < \epsilon$ with $|\n \tilde{h}| \leq \Lip(h) + r$. Therefore, if $h \in {\Lip}_1(M)$ then
\begin{align}
&|\E[h(W)] - \mu_\psi(h)|\\[2mm]
\leq\,& |\E[(h-\tilde{h})(W)]| + |\E[\tilde{h}(W)] - \mu_\psi(\tilde{h})|\\[2mm]
& + |\mu_{\psi}(\tilde{h}) - \mu_{\psi}(h)|\\[2mm]
\leq\,& 2\|h - \tilde{h}\|_\infty + |\E[\tilde{h}(W)] - \mu_\psi(\tilde{h})|\\[2mm]
\leq\,& 2\epsilon + C(1+r)( \E[|R_1|] + \E[|R_2|] + \frac{1}{\lambda}\E[|\delta|^3(|\log|\delta||\vee 1)])
\end{align}
for all $\epsilon,r >0$, from which the result follows.
\end{proof}

\begin{proposition}\label{prop:steinest3}
Suppose $K > 0$ with $\Ric_\psi \geq 2K$ and $h \in C^2_b(M)$. Suppose
\begin{equation}
\begin{aligned}
&\|R\|_\infty < \infty,\\
&\|\n R\|_\infty < \infty,
\end{aligned}\quad
\begin{aligned}
&\|\nabla {\Ric}^\sharp_\psi + d^\star R -2R(\n \psi)\|_\infty < \infty,\\
&\|\n(\nabla {\Ric}^\sharp_\psi + d^\star R -2R(\n \psi))\|_\infty < \infty\\
\end{aligned}
\end{equation}
and that $f$ a solution to the Stein equation \eqref{eq:stein}. Then there exists a positive constant $c_4$ such that $\|\n\n d f\| \leq  c_4 (\|\n h\|_\infty + \|\Hess h\|_\infty)$.
\end{proposition}

\begin{proof}
This follows from the second part of Theorem \ref{thm:estthree} and formula \eqref{eq:poissoln}.
\end{proof}

Now define
\begin{equation}
\mathcal{H}:= \{h \in C_b^2(M): |\n h| \leq 1, |\Hess h| \leq 1\}
\end{equation}
and for probability measures $\mu$ and $\nu$, define the distance
\begin{equation}\label{eq:ddist}
d_{\mathcal{H}}(\mu, \nu) := \sup \{ |\mu(h)- \nu(h)|: h \in \mathcal{H}\}
\end{equation}
whenever the right-hand side exists.

\begin{theorem}\label{thm:mainthmtwo}
Suppose $K>0$ with ${\Ric}_\psi \geq 2K$. Suppose
\begin{equation}
\begin{aligned}
&\|R\|_\infty < \infty,\\
&\|\n R\|_\infty < \infty,
\end{aligned}\quad
\begin{aligned}
&\|\nabla {\Ric}^\sharp_\psi + d^\star R -2R(\n \psi)\|_\infty < \infty,\\
&\|\n(\nabla {\Ric}^\sharp_\psi + d^\star R -2R(\n \psi))\|_\infty < \infty.\\
\end{aligned}
\end{equation}
Then there exists a positive constant $C$ such that for all $\lambda >0$ and all pairs $W,W'$ of identically distributed random variables taking values in $M$ and satisfying (A0), with $R_1$ and $R_2$ defined by (A1) and (A2), we have
\begin{equation}
d_{\mathcal{H}}(\mathcal{L}(W),\mu_\psi) \leq C \left(\frac{1}{\lambda}\E[d^3(W,W')] + \E[|R_1|] + \E[|R_2|]\right)
\end{equation}
where the distance $d_{\mathcal{H}}$ is defined by \eqref{eq:ddist}.
\end{theorem}

\begin{proof}
The proof follows similar lines to that of Theorem \ref{thm:mainthmone}, except that simply we use Proposition \ref{prop:steinest3} to estimate the third derivative.
\end{proof}

Note that the constants appearing in these theorems can all be made explicit, since the estimates derived in Section \ref{sec:derests} can be so too.

\section{Uniform approximation}\label{sec:7}

Now let us consider the problem of uniform approximation on a compact manifold. In this case the curvature assumption $\Ric \geq 2K>0$ can be dispensed with, although we will, for simplicity assume $\psi = 0$. We will denote by $\bar{\mu}$ the normalized volume measure, meaning
\begin{equation}
d \bar{\mu} = \frac{1}{\mu(M)} d\mu.
\end{equation}
All $L^p$-norms $\|\cdot \|_p$ will be defined with respect to the probability measure $\bar{\mu}$. Note that for a smooth function $f$ the harmonic projection $Hf$ of $f$ is given by the integral
\begin{equation}
Hf := \bar{\mu}(f) = \int_M f d\bar{\mu}.
\end{equation}
With this in mind we set
\begin{equation}\label{eq:rayleigh}
\lambda:= \min \bigg\lbrace \frac{\frac{1}{2}\| \nabla f\|_2^2}{\|f\|_2^2} : f \in C^\infty(M), f \neq 0, Hf = 0\bigg\rbrace.
\end{equation}
The point here is that $\lambda>0$ since $M$ is compact. This otherwise need not be the case, even if $M$ is of finite volume. The quantity $\lambda$ coincides with the first positive eigenvalue of the operator $-\frac{1}{2}\Delta$. This quantity is known as the \textit{spectral gap}. For the following theorem, note that if $U$ is a uniformly distributed $M$-valued random variable with $f$ a bounded measurable function then
\begin{equation}
\Var[f(U)] = \int_M f^2 d\bar{\mu} - \left(\int_M f d\bar{\mu}\right)^2,
\end{equation}
where by $\Var[f(U)]$ we mean the variance of the real-valued random variable $f(U)$.

\begin{theorem}\label{thm:compact}
Suppose $M$ is a compact Riemannian manifold with $U$ a uniformly distributed $M$-valued random variable. Then
\begin{align}
\|(P_t - H)f\|_2 \leq e^{-\lambda t} \sqrt{\Var[f(U)]}
\end{align}
for all $t \geq 0$ and bounded measurable $f$.
\end{theorem}

\begin{proof}
First note for $t>0$ that $P_tf$ is smooth with
\begin{equation}
\int_M (P_t -H)f d\bar{\mu} = 0
\end{equation}
and furthermore
\begin{align}
\frac{d}{dt} \| (P_t-H)f\|^2_2 =\,& 2 \int_M (P_t -H)f \frac{d}{dt}P_t f d\bar{\mu}\\
=\,& \int_M (P_t -H)f \Delta P_t f d\bar{\mu}\\
=\,& -\int_M |\nabla P_t f|^2 d\bar{\mu}\\
=\,& -\int_M |\nabla(P_t-H)f|^2 d\bar{\mu}\\
=\,& - \| \nabla (P_t-H)f\|^2_2.
\end{align}
Therefore, by the definition of $\lambda$ and Gronwall's inequality, we see that
\begin{align}
\|(P_t - H)f\|_2^2 \leq e^{-2\lambda t} \|f - Hf\|_2^2 = e^{-2\lambda t} \Var[f(U)]
\end{align}
as required.
\end{proof}

The Rayleigh characterization of $\lambda$, given above by \eqref{eq:rayleigh}, implies that
\begin{equation}\label{eq:poincare}
\|f - Hf\|_2^2 \leq \tfrac{1}{2\lambda}\|\n f\|_2^2
\end{equation}
for any $f$ belonging to the Sobolev space $H_1^2(M)$. Inequality \eqref{eq:poincare} is a particular case of the \textit{Poincar\'{e} inequality}. Together with Theorem \ref{thm:compact} this immediately yields the following corollary:

\begin{corollary}\label{cor:poincareapp}
For all $f\in C^1(M)$ we have
\begin{equation}
\|(P_t - H)f\|_2 \leq e^{-\lambda t} \sqrt{\tfrac{1}{2\lambda}}\|\n f\|_2
\end{equation}
for all $t \geq 0$.
\end{corollary}

Now, the integral kernel $p: (0,\infty) \times M \times M \rightarrow \R$ of the heat semigroup $\lbrace P_t:t \geq 0\rbrace$ is called the \textit{heat kernel}. It is the fundamental solution to the heat equation and coincides with the transition densities of Brownian motion, in the sense that if $X(x)$ is a Brownian motion on $M$ starting at $x$ with $A$ a Borel subset of $M$ then
\begin{equation}
\P\lbrace X_t(x) \in A\rbrace = \int_A p_t(x,y)d\mu(y)
\end{equation}
for each $t>0$. The heat kernel is a smooth, strictly positive function, symmetric in its space variables, which due the stochastic completeness of compact manifolds satisfies the property
\begin{equation}
\int_M p_t(x,y) d\mu(y) = 1
\end{equation}
for each $t>0$ and $x \in M$. Varadhan's asymptotic relation states for a complete Riemannian manifold that
\begin{equation}
\lim_{t\downarrow 0} t \log p_t(x,y) = -\frac{d^2(x,y)}{2}
\end{equation}
uniformly on compact subsets of $M \times M$. Indeed, since for small times the mass of the heat kernel localizes, with Riemannian manifolds being locally Euclidean, this is the sense in which the heat kernel asymptotically approximates the Euclidean kernel. For our purposes, it suffices to note that for each $m \in \N \cup \lbrace 0 \rbrace$ and $\epsilon >0$ the compactness of $M$ implies
\begin{equation}
\sup_{x \in M} \| \n^m p_\epsilon(x,\cdot)\|_2 < \infty
\end{equation}
where the covariant derivatives are taken in the first variable $x$ and the integration of the norm in the second.

\begin{proposition}\label{prop:compderestszero}
Suppose $M$ is a compact Riemannian manifold. For each $\epsilon>0$ there exists a positive constant $C_0(\epsilon)>0$ such that
\begin{align}
\|(P_t-H)f\|_\infty \leq C_0(\epsilon) e^{-\lambda t} \| \n f \|_\infty
\end{align}
for all $t>\epsilon$ and all $f\in C^1$.
\end{proposition}

\begin{proof}
For each $x \in M$ we see that
\begin{align}
(P_t -H)f(x) =\,& (P_\epsilon P_{t-\epsilon} -H)f(x)\\
=\,& \int_M p_\epsilon(x,y) (P_{t-\epsilon} - H)f(y)d \mu(y)
\end{align}
which implies
\begin{align}
|P_t - H|f(x) \leq \mu(M) \| p_\epsilon(x,\cdot)\|_2 \|(P_{t-\epsilon} - H)f\|_2
\end{align}
which by Corollary \ref{cor:poincareapp} yields the inequality since $\| \n f\|_2 \leq \| \n f\|_\infty$.
\end{proof}

By combining Proposition \ref{prop:compderestszero} with Theorem \ref{thm:exponentialconv}, using the latter to cover the case $t \in [0,\epsilon]$, we deduce that if $f$ is the solution to the Stein equation
\begin{equation}
\mathcal{A}f = h - \bar{\mu}(h)
\end{equation}
where $h\in \Lip_1(M)$ then $f$ is given by the formula
\begin{equation}\label{eq:poissoncompact}
f(x) = - \int_0^\infty P_th(x) - \bar{\mu}(h) dt
\end{equation}
for all $x \in M$. We can similarly prove exponential decay of the supremum norms of the derivatives of $P_tf$, as in the following proposition:

\begin{proposition}\label{prop:compderests}
Suppose $M$ is a compact Riemannian manifold. For each $m \in \N$ and $\epsilon>0$ there exists a positive constant $C_m(\epsilon)>0$ such that
\begin{align}
\|\n^m P_tf\|_\infty \leq C_m(\epsilon) e^{-\lambda t} \| \n f \|_\infty
\end{align}
for all $t>\epsilon$ and all $f\in C^1$.
\end{proposition}

\begin{proof}
Firstly note that
\begin{align}
P_tf(x) =\,& P_\epsilon P_{t-\epsilon} f(x)\\
=\,& \int_M p_\epsilon(x,y) P_{t-\epsilon} f(y) d \mu(y)\\
=\,& \int_M p_\epsilon(x,y)(P_{t-\epsilon} - H)f(y)d\mu(y) + Hf
\end{align}
and therefore for each $v \in T_xM$ we have
\begin{equation}
(\n^m P_tf)_x(v) = \int_M (\n^m p_\epsilon(\cdot,y))_x(v) (P_{t-\epsilon} - H)f(y)d\mu(y).
\end{equation}
Consequently
\begin{align}
|\n^m P_tf|(x) \leq \mu(M) \| \n^m p_\epsilon(x,\cdot)\|_2 \| (P_{t-\epsilon} - H)f\|_2
\end{align}
from which the result follows, by Corollary \ref{cor:poincareapp}.
\end{proof}

The point here is that the derivatives of the semigroups all decay exponentially fast to zero and at the same rate at which the semigroup itself converges to equilibrium. By combining Proposition \ref{prop:compderests} with Theorems \ref{thm:estone}, \ref{thm:esttwo} and \ref{thm:estthree}, using those Theorems to cover the case $t \in (0,\epsilon]$, we have analogues of the derivative estimates used in the previous section:

\begin{theorem}\label{thm:estonecompact}
Suppose $M$ is a compact Riemannian manifold with $f \in C^1(M)$. Then there exist positive constants $C_0',C_1',C_2'$ and $C_3'$ such that
\begin{equation}
\|\n P_tf\|_\infty \leq C_0' e^{-\lambda t} \|\n f\|_\infty
\end{equation}
for all $t \geq 0$ with
\begin{equation}
\|\n d P_tf\|_\infty \leq  \frac{C_1' e^{-\lambda t}}{\sqrt{1\wedge t}} \|\n f\|_\infty 
\end{equation}
and
\begin{equation}
\|\n\n d P_tf\|_\infty \leq  \frac{C_2' e^{-\lambda t}}{1\wedge t} \|\n f\|_\infty
\end{equation}
and
\begin{equation}
\|\n\n d P_tf\|_\infty \leq  \frac{C_3' e^{-\lambda t}}{\sqrt{1\wedge t}} (\|\n f\|_\infty + \|\Hess f\|_\infty)
\end{equation}
for all $t > 0$.
\end{theorem}

Following the arguments of the previous section, including the proofs of Theorems \ref{thm:mainthmone} and \ref{thm:mainthmtwo}, using now the estimates provided by Theorem \ref{thm:estonecompact}, we obtain the following versions of our main results, for uniform approximation:

\begin{theorem}\label{thm:mainthmonecompact}
Suppose $M$ is a compact Riemannian manifold. Then there exists a positive constant $C$ such that for all $\lambda >0$ and all pairs $W,W'$ of identically distributed random variables taking values in $M$ and satisfying (A0), with $R_1$ and $R_2$ defined by (A1) and (A2), we have
\begin{equation}
\mathcal{W}(\mathcal{L}(W),\bar{\mu}) \leq C' \left(\frac{1}{\lambda}\E[|\delta|^3(|\log|\delta||\vee 1)] + \E[|R_1|] + \E[|R_2|]\right)
\end{equation}
where $|\delta| = d(W,W')$.
\end{theorem}

\begin{theorem}\label{thm:mainthmtwocompact}
Suppose $M$ is a compact Riemannian manifold. Then there exists a positive constant $C$ such that for all $\lambda >0$ and all pairs $W,W'$ of identically distributed random variables taking values in $M$ and satisfying (A0), with $R_1$ and $R_2$ defined by (A1) and (A2), we have
\begin{equation}
d_{\mathcal{H}}(\mathcal{L}(W),\bar{\mu}) \leq C' \left(\frac{1}{\lambda}\E[d^3(W,W')] + \E[|R_1|] + \E[|R_2|]\right)
\end{equation}
where the distance $d_{\mathcal{H}}$ is defined by \eqref{eq:ddist}.
\end{theorem}

\section{Examples of Riemannian measures}\label{sec:8}

\textbf{Euclidean space.} If $M=\R^n$ with the standard Euclidean inner product then all curvature operators are identically zero and the assumptions of Theorems \ref{thm:mainthmone} and \ref{thm:mainthmtwo} become $\Hess \psi \geq K >0$ with
\begin{equation}
\|\n \Hess \psi\|_\infty < \infty, \quad \|\n \n \Hess \psi\|_\infty < \infty.
\end{equation}
For a symmetric positive definite matrix $A$ and $y \in \R^n$, the function
\begin{equation}
\psi(x,y) = - \tfrac{1}{2}\< x-y, A(x-y)\>
\end{equation}
satisfies these assumptions with the corresponding Gaussian measure given by
\begin{equation}
d\mu_\psi = \left(\frac{\det(A)}{2 \pi}\right)^{\frac{n}{2}} e^{\psi(x-y)} dx.
\end{equation}
If $A = \tfrac{1}{t}\id_{\R^n}$ then the density on the right-hand side becomes the heat kernel
\begin{equation}
p_t(x,y) = {\left( 2 \pi t \right) }^{-\frac{n}{2}} \exp \left(-\frac{d^2(x,y)}{2t} \right)
\end{equation}
for $x,y \in \mathbb{R}^n$ and $t>0$.
\bigbreak
\textbf{Hyperbolic space.}
Denote by $\mathbb{H}^3_\kappa$ the $3$-dimensional hyperbolic space with constant sectional curvatures equal to $\kappa<0$. On this space there is an explicit formula for the heat kernel. In particular
\begin{equation}
p_t(x,y) = (2\pi t)^{-\frac{3}{2}}\exp\left[-\frac{d^2(x,y)}{2t}\right]\frac{\sqrt{-\kappa}d(x,y)e^{\frac{\kappa t}{2}}}{\sinh \left(\sqrt{-\kappa}d(x,y)\right)}
\end{equation}
for $x,y \in \mathbb{H}^{3}_\kappa$ and $t>0$. Denoting $r = d(x,y)$ and differentiating in $x$ we find
\begin{equation}
\n \log p_t(x,y) = -\frac{1}{2t}\n r^2 + \left( \frac{1}{r} - \sqrt{-\kappa} \coth(\sqrt{-\kappa} r)\right) \n r
\end{equation}
so that
\begin{align}
\Hess \log p_t(x,y) =\, -\frac{1}{2t}\Hess r^2 &- \left(\frac{1}{r^2} + \kappa \cosech^2(\sqrt(-\kappa) r)\right) \n r \otimes \n r\\
& + \left(\frac{1}{r} - \sqrt{- \kappa} \coth(\sqrt{-\kappa r})\right) \Hess r
\end{align}
and therefore, for any vector field $X$, we see that
\begin{align}
(\Ric - 2 \Hess \log p_t(x,y))(X,X) \\
=\, 2\kappa |X|^2 + \frac{1}{t}\Hess r^2(X,X) &+ 2\left(\frac{1}{r^2} + \kappa \cosech^2(\sqrt(-\kappa) r)\right) \< \n r, X\>^2\\
& + 2\left(\frac{1}{r} - \sqrt{- \kappa} \coth(\sqrt{-\kappa r})\right) \Hess r(X,X).
\end{align}
But it is well known that in this setting we have
\begin{align}
&\Hess r^2(X,X) = 2\< \n r, X\>^2 + \frac{2rH'(r)}{H(r)}(|X|^2 - \<\n r,X\>^2), \label{eq:hessr2}\\
&\Hess r(X,X) = \frac{H'(r)}{H(r)}(|X|^2 - \< \n r, X\>^2) \label{eq:hessr}
\end{align}
where the function $H:[0,\infty) \rightarrow [0,\infty)$ is given by
\begin{equation}
H(r) = \frac{\sinh(\sqrt{-\kappa}r)}{\sqrt{-\kappa}}
\end{equation}
for $r \geq 0$. Using equation \eqref{eq:hessr} and by separately considering the cases $X \perp \n r$ and $X \not\perp \n r$ it is easy to see that
\begin{equation}
(\Ric - 2 \Hess \log p_t(x,y))(X,X) \geq 2\kappa |X|^2 + \frac{1}{t}\Hess r^2(X,X)
\end{equation}
which by \eqref{eq:hessr2} implies
\begin{equation}\label{eq:hypercond}
\Ric - 2 \Hess \log p_t(x,y) \geq 2\left(\kappa + \frac{1}{t}\right).
\end{equation}
Note that the right-hand side of \eqref{eq:hypercond} is strictly positive so long as
\begin{equation}
t < -\frac{1}{\kappa}.
\end{equation}
Fixing such a $t$, by Theorem \ref{thm:exponentialconv}, the distribution of a diffusion with generator
\begin{equation}
\frac{1}{2}\Delta + \n \log p_t(x,y)
\end{equation}
will therefore converge exponentially fast in the Wasserstein distance to the invariant probability measure
\begin{equation}
p_t(x,y)d\mu(x).
\end{equation}
In this example the Hessian of the logarithm of the heat kernel balances the negative curvature on $\mathbb{H}^3_\kappa$ so long as $t < -1/\kappa$ and therefore heat kernel approximation is possible in this setting. Note that as the curvature $\kappa$ approaches zero, the bound on $t$ becomes $t<\infty$ as it was on $\R^3$.
\bigbreak
\textbf{Spherical space.}
The uniform probability measure $\bar{\mu}$ on a compact manifold is an example of a Riemannian measure approximations to which are covered by Theorems \ref{thm:mainthmonecompact} and \ref{thm:mainthmtwocompact}. The sphere $\mathbb{S}^n_{\kappa}$ of dimension $n$ with constant sectional curvatures $\kappa >0$, that is, with radius $\pi/\sqrt{\kappa}$, is one such compact manifold. Note that on $\mathbb{S}^n_{\kappa}$ the Ricci curvature is given by $\Ric = (n-1)\kappa >0$ and therefore the uniform probability measure on $\mathbb{S}^n_{\kappa}$ also satisfies the conditions of Theorems \ref{thm:mainthmone} and \ref{thm:mainthmtwo}, which provide the same estimates as Theorems \ref{thm:mainthmone} and \ref{thm:mainthmtwo} except with slightly different constants. Recall that by the Bonnet-Myers theorem, if $M$ is a Riemannian manifold with $\Ric \geq (n-1)\kappa>0$ then $M$ is compact and has diameter at most $\pi/\sqrt{\kappa}$.


\bigbreak
\textbf{Diffusion operators.} We noted earlier that on an arbitrary smooth manifold $M$, any smooth elliptic diffusion operator $\A$ induces a Riemannian metric on $M$ with respect to which it takes the form $\A=\tfrac{1}{2}\Delta + Z$ for some smooth vector field $Z$. For example, suppose $\sigma : \R^n \times \R^n \rightarrow \R^n$ is a smooth map such that $\sigma(x):\R^n \rightarrow \R^n$ is a linear bijection for each $x \in \R^n$ with $\psi$ a smooth vector field on $\R^n$. Then, with respect to the standard coordinates $\{x^i\}_{i=1}^n$ on $\R^n$, the infinitesimal generator of the It\^{o} diffusion
\begin{equation}
dX_t = \sigma(X_t)dB_t + \n \psi(X_t)dt
\end{equation}
is given by the operator
\begin{equation}
\A f(x) = \frac{1}{2}\sum_{i,j=1}^n (\sigma(x) \sigma(x)^{T})_{ij}(x) \frac{\partial^2}{\partial x_i \partial x_j} f(x) + \sum_{i=1}^n (\n \psi)_i(x) \frac{\partial}{\partial x_i} f(x)
\end{equation}
which induces a metric $g$ on $\R^n$ with components given by
\begin{equation}
g_{ij} = ((\sigma(x) \sigma(x)^{T})^{-1})_{ij}.
\end{equation}
The Christoffel symbols are given by
\begin{equation}
\Gamma^{m}{}_{ij}={\frac{1}{2}}\,g^{mk}\left({\frac {\partial }{\partial x^{j}}}g_{ki}+{\frac {\partial }{\partial x^{i}}}g_{kj}-{\frac {\partial }{\partial x^{k}}}g_{ij}\right)
\end{equation}
where
\begin{equation}
g^{mk} = (\sigma(x) \sigma(x)^{T})_{mk}
\end{equation}
and the components of the Ricci curvature tensor are given by
\begin{equation}
R_{ij}={\frac {\partial \Gamma ^{\ell }{}_{ij}}{\partial x^{\ell }}}-{\frac {\partial \Gamma ^{\ell }{}_{i\ell }}{\partial x^{j}}}+\Gamma ^{m}{}_{ij}\Gamma ^{\ell }{}_{\ell m}-\Gamma ^{m}{}_{i\ell }\Gamma ^{\ell }{}_{jm}.
\end{equation}
The components of the Hessian of $\psi$ are given by
\begin{equation}
(\Hess \psi)_{ij}=\left({\frac {\partial ^{2}\psi}{\partial x^{i}\partial x^{j}}}-\Gamma^{k}{}_{ij}{\frac {\partial \psi}{\partial x^{k}}}\right)
\end{equation}
so positivity of the Bakry-Emery-Ricci tensor can, in principle, be checked by direct computation. Explicit formulas for the other curvature operators are also well known, so the other assumptions in Theorems \ref{thm:mainthmone} and \ref{thm:mainthmtwo} can be checked similarly. If all necessary assumptions are verified then the invariant probability measure would take the form
\begin{equation}
d\mu_{\psi}(x) = \frac{1}{c} \cdot \frac{e^{\psi(x)}}{|\det(\sigma(x))|}\,dx
\end{equation}
for some normalizing constant $c$. 

\section{Commutation relations}\label{sec:comm}

Suppose $M$ is a Riemannian manifold with $\n$ the Levi-Civita connection. In this section $X,Y,V$ and $W$ will denote fixed vector fields on $M$.
The covariant derivative of a tensor field $T$ is denoted $\nabla T$ and we will adopt the convention
\begin{equation}
\n_X T = (\n T)(X,\ldots)
\end{equation}
so that $X$ appears as the first entry rather than the last. The second covariant derivative of a tensor field $T$ is defined by
\begin{equation}
\n^2_{X,Y}T = (\n \n T)(X,Y,\ldots) = \n_X \n_YT - \n_{\n_X Y}T
\end{equation}
and similarly the third covariant derivative is defined by
\begin{equation}
\n^3_{X,Y,V}T = (\n \n \n T)(X,Y,V,\ldots) = \n_X \n^2_{Y,V}T - \n^2_{\n_X Y,V} T - \n^2_{Y,\n_X V} T.
\end{equation}
The Riemannian curvature tensor $R$ is defined by
\begin{align}
R(X,Y) := \n_X \n_Y  - \n_Y \n_X  - \n_{[X,Y]} = \n^2_{X,Y} - \n^2_{Y,X}
\end{align}
where $[X,Y]$ denotes the Lie bracket of the vector fields $X$ and $Y$, we set
\begin{equation}
R(X,Y,V,W) := \langle R(X,Y)V,W\rangle
\end{equation}
and define the Ricci curvature tensor $\Ric$ by
\begin{equation}
\Ric(X,Y) = \tr R(X,\cdot,\cdot,Y).
\end{equation}
If $f$ is a smooth function then the Hessian of $f$ is the second covariant derivative of $f$ and is therefore given by
\begin{equation}
(\Hess f)(X,Y) := (\n df)(X,Y) = X(Yf) - df(\n_X Y ).
\end{equation}
The Hessian of $f$ is symmetric since the Levi-Civita connection is torsion free.

\begin{lemma}\label{lem:commutator1}
For any vector fields $X$ and $Z$ we have
\begin{equation}
d(Z(f))(X) = (\n_Z df)(X) + df(\n_X Z).
\end{equation}
\end{lemma}

\begin{proof}
Since $\Hess f$ is symmetric it follows that
\begin{equation}
d(Z(f))(X) = \n_X (df(Z)) = (\n_X df)(Z) + df(\n_X Z) = (\n_Z df)(X) + df(\n_X Z)
\end{equation}
as required.
\end{proof}

The second covariant derivative is not symmetric in general. In particular, if $T \in \Gamma( (T^\star M)^{\otimes k})$ and $V_1,V_2,\ldots,V_k$ are vector fields then the commutator is given by the \emph{Ricci identity} which states
\begin{align}
(\n^2_{X,Y}T - \n^2_{Y,X}T)(V_1,V_2,\ldots,V_k)=\,& -T(R(X,Y)V_1,V_2,\ldots,V_k) \\
&\quad-T(V_1,R(X,Y)V_2,\ldots,V_k)\\
&\quad\quad- \ldots - T(V_1,V_2,\ldots,R(X,Y)V_k).
\end{align}
Using the Ricci identity we can calculate commutators involving covariant differentiation in the direction of the vector field $Z$.

\begin{lemma}\label{lem:commutator2}
For any vector fields $X,Y$ and $Z$ we have
\begin{equation}
(\n_X \n_Z df)(Y)= (\n_Z \n df)(X,Y) + (df)(R(Z)(X,Y)) + (\n df)(\n_X Z,Y)
\end{equation}
where $R(Z)(X,Y):=R(Z,X)Y$.
\end{lemma}

\begin{proof}
By the definition of the second covariant derivative and the Ricci identity we see that
\begin{align}
(\n_X \n_Z df)(Y) =\,& (\n^2_{X,Z} df)(Y) + (\n_{\n_X Z} df)(Y)\\
=\,& (\n_{Z,X}^2 df)(Y) - (df)(R(X,Z)Y) + (\n_{\n_X Z} df)(Y)\\
=\,& (\n_Z \n df)(X,Y) + (df)(R(Z,X)Y) + (\n df)(\n_X Z,Y)
\end{align}
as required.
\end{proof}



\begin{lemma}\label{lem:commutator3}
For any vector fields $X,Y,V$ and $Z$ we have
\begin{align}
(\n_X \n_Z \n df)(Y,V) =\,& (\n_{Z}\n \n df)(X,Y,V) + (\n df)(R(Z)(X,Y),V)\\
&+ (\n df)(Y,R(Z)(X,V)) + (\n \n df)(\n_X Z,Y,V)
\end{align}
where $R(Z)$ is defined as in Lemma \ref{lem:commutator2}.
\end{lemma}

\begin{proof}
As in the proof of Lemma \ref{lem:commutator3}, using the Ricci identity we see that
\begin{align}
(\n_X \n_Z \n df)(Y,V) =\,&  (\n^2_{X,Z} \n df)(Y,V) + (\n_{\n_X Z} \n df)(Y,V)\\
=\,& (\n_{Z,X}^2 \n df)(Y,V) - (\n df)(R(X,Z)Y,V)\\
&- (\n df)(Y,R(X,Z)V) + (\n \n df)(\n_X Z,Y,V)\\
=\,& (\n_{Z}\n \n df)(X,Y,V) + (\n df)(R(Z,X)Y,V)\\
&+ (\n df)(Y,R(Z,X)V) + (\n \n df)(\n_X Z,Y,V)
\end{align}
as required.
\end{proof}

Lemmas \ref{lem:commutator1}, \ref{lem:commutator2} and \ref{lem:commutator3} show how the first order part of the operator $\A$, namely covariant differentiation in direction $Z$, commutes with covariant differentiation in the direction of some other vector field. We will also need commutation formulas for the second order part, namely the Laplacian.

\section{Weitzenb\"{o}ck formulas}\label{sec:10}

Suppose that $E \rightarrow M$ is a vector bundle over $M$. Suppose $E$ is equipped with a covariant derivative $\n$ with curvature tensor $R^E$ and set
\begin{equation}
\square := \tr \n^2.
\end{equation}
The divergence operator acting on $R^E$ is defined by
\begin{equation}
(\n \cdot R^E)_X := \tr (\n_{\cdot} R^E)(\cdot,X)
\end{equation}
for each vector field $X$. In terms of this operator, the commutator of $\square$ and $\n$ is given by \cite[Proposition~A.10]{DriverThalmaier2001}, which states that if $a \in \Gamma(E)$ and $X$ is a vector field then
\begin{equation}\label{eq:weitgen}
\n_X \square a =  (\square \n a)(X) -(\n a) ({\Ric}^\sharp (X)) - (\n \cdot R^E)_X a + 2 \tr R^E(X,\cdot) \n_{\cdot} a.
\end{equation}
The following lemma is called the \emph{Weitzenb\"{o}ck formula}.

\begin{lemma}\label{lem:weitzenbock1}
We have
\begin{equation}
d \Delta f = \tr \nabla^2 df - (df)({\Ric}^\sharp).
\end{equation}
\end{lemma}

\begin{proof}
This follows from equation \eqref{eq:weitgen} by considering the case where $E = C^\infty(M)$, since in this case $R^E = 0$ and $\n^E=d$.
\end{proof}

Consider the operator $d^\star R$ defined by
\begin{equation}\label{eq:dstarR}
(d^\star R)(v_1,v_2) := -\tr \nabla_\cdot R(\cdot,v_1)v_2
\end{equation}
which satisfies
\begin{equation}
\langle d^\star R(v_1,v_2),v_3\rangle =  \langle (\nabla_{v_3} {\Ric}^\sharp) (v_1),v_2\rangle- \langle (\nabla_{v_2} {\Ric}^\sharp) (v_3),v_1\rangle
\end{equation}
for all $v_1,v_2,v_3 \in T_xM$ and $x \in M$. In particular, this operator vanishes on Ricci parallel manifolds. In terms of $d^\star R$, the next lemma expresses our second Weitzenb\"{o}ck-type formula.

\begin{lemma}\label{lem:weitzenbock2}
For any vector field $X$ we have
\begin{align}
\n_X \square df =\,& (\square \n df)(X) -(\n df)({{\Ric}^\sharp (X)})\\
& - (df)((d^\star R) (X,\cdot)) + 2 \tr  (\n_{\cdot} df)(R(\cdot,X))
\end{align}
where $d^\star R$ is defined by \eqref{eq:dstarR}.
\end{lemma}

\begin{proof}
If $a \in \Gamma(T^\star M)$ with $X$ a vector field then by equation \eqref{eq:weitgen} with $E = T^\star M$ we have
\begin{equation}
\n_X \square a =  (\square \n a)(X) -(\n a)({\Ric}^\sharp (X)) - (\n \cdot R^{T^\star M})_X a + 2 \tr R^{T^\star M}(X,\cdot) \n_{\cdot} a.
\end{equation}
Moreover if $Y$ and $V$ are vector fields then $R^{T^\star M}$ satisfies
\begin{equation}
(R^{T^\star M}(X,Y) a)(V) = - a(R(X,Y)V)
\end{equation}
and therefore
\begin{equation}
(\n \cdot R^{T^\star M})_X a= \tr (\n_{\cdot} R^{T^\star M})(\cdot,X) a= -a(\tr (\n_{\cdot} R)(\cdot,X))
\end{equation}
and
\begin{equation}
\tr R^{T^\star M}(X,\cdot) \n_{\cdot} a = -\tr (\n_{\cdot} a)(R(X,\cdot))
\end{equation}
and therefore
\begin{align}
\n_X \square a =\,&  (\square \n a)(X) -(\n a)({{\Ric}^\sharp (X)})\\
&+ a(\tr (\n_{\cdot} R)(\cdot,X)) - 2 \tr  (\n_{\cdot} a)(R(X,\cdot)).
\end{align}
The result follows from this using the definition of $d^\star R$ and by setting $a = df$.
\end{proof}

Continuing, we next apply equation \eqref{eq:weitgen} to the case $E = T^\star M \otimes T^\star M$.

\begin{lemma}\label{lem:weitzenbock3}
For any vector fields $X,V$ and $W$ we have
\begin{align}
&(\n_X \square \n df)(V,W)\\
=\,& (\square \n \n df)(X,V,W) - (\n \n df) ({\Ric}^\sharp (X),V,W)\\
&  -\n df ( (d^\star R)(X,V),W) - \n df (V,(d^\star R)(X,W))\\
& -2\tr (\n_\cdot \n df)(R(X,\cdot)V,W)- 2\tr (\n_\cdot \n df)(V,R(X,\cdot)W)
\end{align}
where $d^\star R$ is defined by \eqref{eq:dstarR}.
\end{lemma}

\begin{proof}
If $b \in \Gamma(T^\star M \otimes T^\star M)$ with $X,Y,V$ and $W$ vector fields then
\begin{equation}
(R^{T^\star M \otimes T^\star M}(X,Y) b)(V,W) = -b(R(X,Y)V,W) - b(V,R(X,Y)W)
\end{equation}
and consequently
\begin{align}
((\n \cdot R^{T^\star M \otimes T^\star M})_X b)(V,W) =\,& (\tr (\n_{\cdot} R^{T^\star M \otimes T^\star M})(\cdot,X)b)(V,W)\\
=\,& -\tr b ( \n_{\cdot} R(\cdot,X)V,W)-\tr b (V,\n_{\cdot} R(\cdot,X)W)
\end{align}
and also
\begin{align}
(\tr R^{T^\star M \otimes T^\star M}(X,\cdot) \n_{\cdot} b)(V,W) =\,& - \tr (\n_\cdot b)(R(X,\cdot)V,W)\\
&- \tr (\n_\cdot b)(V,R(X,\cdot)W).
\end{align}
Therefore equation \eqref{eq:weitgen} becomes
\begin{align}
(\n_X \square b)(V,W) =\,&  (\square \n b)(X,V,W) -(\n b) ({\Ric}^\sharp X,V,W)\\
& - ((\n \cdot R^{T^\star M \otimes T^\star M})_X b)(V,W)\\
& + 2 (\tr R^{T^\star M \otimes T^\star M}(X,\cdot) \n_{\cdot} b)(V,W)\\
=\,& (\square \n b)(X,V,W) - (\n b) ({\Ric}^\sharp X,V,W)\\
&  +\tr b ( \n_{\cdot} R(\cdot,X)V,W)+\tr b (V,\n_{\cdot} R(\cdot,X)W)\\
& -2\tr (\n_\cdot b)(R(X,\cdot)V,W)- 2\tr (\n_\cdot b)(V,R(X,\cdot)W).
\end{align}
The result follows from this by setting $b = \n df$.
\end{proof}

Now that we have Lemmas \ref{lem:commutator1}, \ref{lem:commutator2}, \ref{lem:commutator3}, \ref{lem:weitzenbock1}, \ref{lem:weitzenbock2} and \ref{lem:weitzenbock3}, we can proceed in the next section to calculate formulas for the first three derivatives of the semigroup generated by $\A$.

\section{Derivative formulas}\label{sec:11}

For $f \in C^1_b(M)$ and $t>0$ set $f_s := P_{t-s}f$ for $s \in [0,t]$. For $u \in T_xM$ denote by $W_s(u)$ the solution, along the paths of $X(x)$, to the covariant ordinary differential equation
\begin{equation}\label{eq:eqfordpt}
D W_s(u) = - \frac{1}{2}{\Ric}^\sharp_Z W_s(u)ds
\end{equation}
with $W_0(u) = u$.

\begin{lemma}\label{lem:Nlocmart}
For $u \in T_xM$ set
\begin{equation}
N_s(u) := df_s(W_s(u)).
\end{equation}
Then $N_s(u)$ is local martingale.
\end{lemma}

\begin{proof}
Using the relations
\begin{align}
d(Zf) &= \n_Z df + df(\n Z),\\
d \Delta f &= \square df - df({\Ric}^\sharp)
\end{align}
given by Lemmas \ref{lem:commutator1} and \ref{lem:weitzenbock1}, It\^{o}'s formula implies
\begin{equation}\label{eq:locmart1}
\begin{split}
dN_s(u) =\,& (\nabla d f_s)(//_sdB_s,W_s(u)) + df_s(DW_s(u))\\
&+\left(\partial_s+\tfrac{1}{2}\square + \nabla_Z\right)(df_s)(W_s(u))ds\\
=\,& (\nabla d f_s)(//_sdB_s,W_s(u)).
\end{split}
\end{equation}
In particular, $N_s(u)$ is a local martingale.
\end{proof}

\begin{theorem}\label{thm:locformone}
Suppose $\Ric_Z$ is bounded below with $f \in C^1_b(M)$. Then
\begin{equation}\label{eq:formulalocone}
(dP_tf)(u)  = \E[(df)(W_t(u))]
\end{equation}
for all $u \in T_xM$, $x\in M$ and $t\geq 0$.
\end{theorem}

\begin{proof}
Since $\Ric_Z$ is bounded below it follows that the local martingale $N_s(u)$ is bounded and therefore a martingale on $[0,t]$. Therefore
\begin{equation}
\E[N_0(u)] = \E[N_t(u)]
\end{equation}
which yields the formula by the definition of $N_s(u)$.
\end{proof}

Next we prove a formula for the second derivative. For each $u,v \in T_{x}M$ define $W'_s(u,v)$ along the paths of $X(x)$ by
\begin{equation}
\begin{split}
W'_s(u,v):=\,& W_s \int_0^s W_r^{-1} R(//_r dB_r,W_r(u))W_r (v)\\
&-\frac{1}{2}  W_s \int_0^s W_r^{-1}(\nabla {\Ric}^\sharp_Z + d^\star R -2R(Z))(W_r (u),W_r(v))dr
\end{split}
\end{equation}
where $W_s$ is defined as in equation \eqref{eq:eqfordpt} with $W_0 = \id_{T_xM}$. Here $d^\star R$ is minus the divergence of $R$ as defined in \eqref{eq:dstarR}. The process $W^{\prime}_s(u,v)$ is the solution to the covariant It\^{o} equation
\begin{equation}\label{eq:eqforw'}
\begin{split}
DW^{\prime}_s(u,v) =\text{ }& R(//_sdB_s,W_s( u ))W_s(v)\\
& -\tfrac{1}{2} (d^\star R -2 R(Z) +\nabla {\Ric}^\sharp_Z) (W_s( u),W_s(v))ds\\
&-\tfrac{1}{2} {\Ric}^\sharp_Z (W^{\prime}_s(u,v))ds
\end{split}
\end{equation}
with $W^{\prime}_0(u,v) = 0$.

\begin{lemma}\label{lem:N'locmart}
For $u,v \in T_xM$ set
\begin{equation}
N^{\prime}_s(u,v) := (\nabla d f_s)(W_s(u),W_s(v)) + (df_s)(W^{\prime}_s (u,v)).
\end{equation}
Then $N^{\prime}_s(u,v)$ is local martingale.
\end{lemma}

\begin{proof}
As in the proof of Lemma \ref{lem:Nlocmart} we have
\begin{align}
\partial_s df_s = d(\partial_s f_s)= -  d\A f_s= -(\tfrac{1}{2}\square + \n_Z) df_s + \tfrac{1}{2}df_s({\Ric}_Z^\sharp)
\end{align}
and therefore, by the relations given by Lemmas \ref{lem:commutator2} and \ref{lem:weitzenbock2}, we have
\begin{align}
&(\partial_s (\nabla d f_s))(v_1,v_2)\\
=\,& (\nabla \partial_s  d f_s)(v_1,v_2)\\
=\,& -(\tfrac{1}{2}\n \square + \n \n_Z)df_s(v_1,v_2) + \tfrac{1}{2}\n_{v_1}(df_s({\Ric}_Z^\sharp))(v_2)\\
=\,& -(\tfrac{1}{2}\n \square + \n \n_Z)df_s(v_1,v_2) + \tfrac{1}{2}(\n_{v_1} df_s)({\Ric}_Z^\sharp (v_2))\\
& + \tfrac{1}{2}(df_s)((\n_{v_1}{\Ric}_Z^\sharp)(v_2))\\
=\,& - ((\tfrac{1}{2}\square + \n_Z) \n df_s)(v_1,v_2) + \tfrac{1}{2}(\n df_s)(v_1,{\Ric}_Z^\sharp (v_2))\\
& + \tfrac{1}{2}(df_s)((\n_{v_1}{\Ric}_Z^\sharp)(v_2))+\tfrac{1}{2}(\n df_s)({\Ric}^\sharp (v_1),v_2)\\
&+ \tfrac{1}{2}(df_s)((d^\star R) (v_1,v_2)) - \tr  (\n_{\cdot} df_s)(R(\cdot,v_1)v_2)\\
&- (df_s)(R(Z)(v_1,v_2)) - (\n df_s)(\n_{v_1} Z,v_2)
\end{align}
at each $x \in M$ with $v_1,v_2 \in T_xM$. Consequently by It\^{o}'s formula we see that
\begin{equation}
\begin{split}
&dN^{\prime}_s(u,v)\\
 =\,& (\nabla_{//_s dB_s} \nabla d f_s)(W_s(u),W_s(v))\\
& + (\nabla d f_s)\left(DW_s(u),W_s(v)\right) + (\nabla d f_s)\left(W_s(u),DW_s(v)\right)\\
& + \left(\partial_s+\tfrac{1}{2}\square + \nabla_Z\right) (\nabla d f_s)(W_s(u),W_s(v))ds\\
& + (\nabla_{//_sdB_s} df_s)(W^{\prime}_s(u,v)) + (df_s)\left(DW^{\prime}_s(u,v)\right) + [ d(df),DW^{\prime}(u,v)]_s\\
& + \left(\partial_s + \tfrac{1}{2}\square + \nabla_Z\right)(df_s)(W^{\prime}_s(u,v))ds\\
=\,& (\nabla_{//_s dB_s} \nabla d f_s)(W_s(u),W_s(v)) + (\nabla_{//_sdB_s} df_s)(W^{\prime}_s(u,v))\\
& + (df_s)(R(//_s dB_s,W_s(u))W_s(v))
\end{split}
\end{equation}
for which we calculated
\begin{equation}
\left[d(df),DW^{\prime}(u,v)\right]_s = \tr (\n_{\cdot} df_s)(R(\cdot,W_s(u))W_s(v))ds.
\end{equation}
It follows that $N_s'(u,v)$ is, in particular, a local martingale.
\end{proof}

\begin{theorem}
If the local martingale $N'_s(u,v)$ is a martingale with $f \in C^2_b(M)$ then
\begin{equation}
(\n d P_t f)(u,v) = \E[ (\n df)(W_t(u),W_t(v)) + (df)(W'_t(u,v))]
\end{equation}
for all $u,v \in T_xM$, $x \in M$ and $t \geq 0$.
\end{theorem}

\begin{proof}
This follows from Lemma \ref{lem:N'locmart} since if $N'_s(u,v)$ is a martingale then
\begin{equation}
\E[N'_0(u,v)] = \E[N'_t(u,v)]
\end{equation}
for all $t \geq 0$.
\end{proof}

Note that $N'_s(u,v)$ is a martingale if, for example, $M$ is compact since in this case all the curvature operators and terms involving $Z$ are necessarily bounded. Our next step is to integrate by parts to reduce the number of derivatives on the right-hand side of the formula by one. Doing so at the level of local martingales, we can localize the contribution of all curvature operators other than the Bakry-Emery-Ricci tensor. In particular, suppose $D$ is a regular domain with $x \in D$ and denote by $\tau$ the first exit time of $X(x)$ from $D$. That $D$ is a regular domain means that $D$ is a connected open domain with compact closure and smooth boundary. The localization is achieved via the introduction of a bounded adapted process $k$ with paths in the Cameron-Martin space $L^{1,2}([0,t];\Aut(T_{x}M))$ such that $k_s=0$ for $s \geq \tau \wedge t$.

\begin{lemma}\label{lem:ibp1}
Suppose $k$ is bounded adapted process with paths in the Cameron-Martin space $L^{1,2}([0,t];\Aut(T_{x}M))$. Then
\begin{align}
&(\nabla d f_s)(W_s(k_s u),W_s(v)) + (df_s)(W^{\prime}_s (k_s u,v)) \\
&- (df_s)(W_s(v))\int_0^s \langle W_r(\dot{k}_ru),//_rdB_r\rangle - \int_0^s  (df_r)(W^{\prime}_r (\dot{k}_r u,v))dr
\end{align}
is a local martingale, for each $u,v \in T_xM$.
\end{lemma}

\begin{proof}
Since
\begin{equation}
d(N^{\prime}_s(k_s u,v)) = N'_s(\dot{k}_su,v)ds + dN'_s(k_su,v)
\end{equation}
it follows, by Lemma \ref{lem:N'locmart}, that
\begin{equation}\label{eq:loclmartprof1}
N^{\prime}_s(k_s u,v) - \int_0^s  (\nabla d f_r)(W_r(\dot{k}_r u),W_r(v)) dr - \int_0^s  (df_r)(W^{\prime}_r (\dot{k}_r u,v))dr
\end{equation}
is a local martingale. By equation \eqref{eq:locmart1} it follows that
\begin{equation}
df_s(W_s(v)) = df_0(v) + \int_0^s (\nabla d f_r)(//_rdB_r,W_r(v))
\end{equation}
which, by integration by parts at the level of local martingales, implies
\begin{equation}\label{eq:loclmartprof2}
\int_0^s  (\nabla d f_r)(W_r(\dot{k}_r u),W_r(v)) dr-df_s(W_s(v))\int_0^s \langle W_r(\dot{k}_ru),//_rdB_r\rangle
\end{equation}
is yet another local martingale. Since \eqref{eq:loclmartprof1} and \eqref{eq:loclmartprof2} are local martingales, consequently
\begin{equation}
N^{\prime}_s(k_s u,v) - df_s(W_s(v))\int_0^s \langle W_r(\dot{k}_ru),//_rdB_r\rangle - \int_0^s  (df_r)(W^{\prime}_r (\dot{k}_r u,v))dr
\end{equation}
is a local martingale, from which the claim follows by the definition of $N'_s(\cdot,v)$.
\end{proof}

\begin{theorem}\label{thm:locformtwo}
Suppose ${\Ric}_Z$ is bounded below with $f \in C^1_b(M)$. Fix $t>0$, suppose $k$ is as in Lemma \ref{lem:ibp1} with $k_0=\id_{T_x M}$ and $k_s=0$ for $s \geq \tau \wedge t$ and that
\begin{equation}
\E\left[\int_0^t |\dot{k}_s|^2 ds\right] < \infty.
\end{equation}
Then
\begin{align}
(\n dP_tf)(u,v) =\,& -\E\left[ (df)(W_t(v))\int_0^t \langle W_s(\dot{k}_s u),//_sdB_s\rangle\right]\\
&- \E\left[ (df)(W_t\int_0^t  W_s^{-1}W^{\prime}_s (\dot{k}_s u,v)ds)\right]
\end{align}
for all $u,v \in T_xM$.
\end{theorem}

\begin{proof}
Since $D$ is compact with $k$ vanishing once $X(x)$ exits $D$ it follows that the local martingale appearing in Lemma \ref{lem:ibp1} is a true martingale on $[0,t]$ and consequently
\begin{align}
(\n dP_tf)(u,v) =\,& -\E\left[ (df)(W_t(v))\int_0^t \langle W_s(\dot{k}_s u),//_sdB_s\rangle\right]\\
&- \E\left[ \int_0^t (df_s)(W^{\prime}_s (\dot{k}_s u,v))ds\right].
\end{align}
Now apply Theorem \ref{thm:locformone} and the Markov property to the second term on the right-hand side.
\end{proof}

Finally we prove a formula for the third derivative. For each $u,v,w \in T_xM$ define $W''_s(u,v,w)$ along the paths of $X(x)$ by
\begin{align}
DW''_s(u,v,w) =\,& (\n_{W_s(u)} R)(//_sdB_s,W_s(v))W_s(w)\\
& + R(//_sdB_s,W'_s(u,v))W_s(w)\\
& + R(//_sdB_s,W_s(v))W'_s(u,w)\\
& + R(//_sdB_s,W_s(u))W'_s(v,w)\\
& -\tfrac{1}{2} \n_{W_s(u)} (d^\star R - 2R(Z) + \n {\Ric}_Z^\sharp)(W_t(v),W_t(w))ds\\
& -\tfrac{1}{2} (d^\star R -2R(Z) + \n {\Ric}_Z^\sharp)(W'_s(u,v),W_s(w))ds\\
& -\tfrac{1}{2} (d^\star R -2R(Z) + \n {\Ric}_Z^\sharp)(W_s(v),W'_s(u,w))ds\\
& -\tfrac{1}{2} (d^\star R -2R(Z) + \n {\Ric}_Z^\sharp)(W_s(u),W'_s(v,w))ds\\
& -\tfrac{1}{2} {\Ric}_Z^\sharp(W''_s(u,v,w))ds\\
& + \tr R(\cdot,W_s(u))R(\cdot,W_s(v))W_s(w) ds
\end{align}
with $W''_0(u,v,w) = 0$. Recall that the processes $W_s$ and $W_s'$ are defined as in \eqref{eq:eqfordpt} and \eqref{eq:eqforw'}.

\begin{lemma}\label{lem:N''locmart}
For $u,v,w \in T_xM$ set
\begin{align}
N''_s(u,v,w) =\,& (\n \n df_s)(W_s(u),W_s(v),W_s(w))\\
& + (\n df_s)(W_s(v),W'_s(u,w))\\
& + (\n df_s)(W'_s(u,v),W_s(w))\\
& + (\n df_s)(W_s(u),W'_s(v,w))\\
& + (df_s)(W''_s(u,v,w)).
\end{align}
Then $N''_s(u,v,w)$ is a local martingale.
\end{lemma}

\begin{proof}
As in the proof of Lemma \ref{lem:N'locmart} we have
\begin{align}
\partial_s df_s = -(\tfrac{1}{2}\square + \n_Z) df_s + \tfrac{1}{2}df_s({\Ric}_Z^\sharp)
\end{align}
and 
\begin{align}
(\partial_s (\nabla d f_s))(v_1,v_2)=\,& - ((\tfrac{1}{2}\square + \n_Z) \n df_s)(v_1,v_2) + \tfrac{1}{2}(\n df_s)(v_1,{\Ric}_Z^\sharp (v_2))\\
& + \tfrac{1}{2}(df_s)((\n_{v_1}{\Ric}_Z^\sharp)(v_2))+\tfrac{1}{2}(\n df_s)({\Ric}_Z^\sharp (v_1),v_2)\\
&+ \tfrac{1}{2}(df_s)((d^\star R - 2R(Z))(v_1,v_2)) - \tr  (\n_{\cdot} df_s)(R(\cdot,v_1)v_2)
\end{align}
and therefore also
\begin{align}
&(\partial_s (\n \n df_s))(v_1,v_2,v_3) \\
=\,& - ((\tfrac{1}{2}\n_{v_1}\square + \n_{v_1} \n_Z) \n df_s)(v_2,v_3) + \tfrac{1}{2}(\n_{v_1} \n df_s)(v_2,{\Ric}_Z^\sharp (v_3))\\
&+ \tfrac{1}{2}(\n_{v_2} df_s)((\n_{v_1} {\Ric}_Z^\sharp )(v_3)) + \tfrac{1}{2}(\n_{v_1} df_s)((\n_{v_2}{\Ric}_Z^\sharp)(v_3))\\
&+ \tfrac{1}{2}( df_s)((\n_{v_1}\n {\Ric}_Z^\sharp)(v_2,v_3))+\tfrac{1}{2}(\n_{v_1}\n df_s)({\Ric}_Z^\sharp (v_2),v_3)\\
&+\tfrac{1}{2}(\n df_s)((\n_{v_1}{\Ric}_Z^\sharp) (v_2),v_3)+ \tfrac{1}{2}(\n_{v_1}df_s)((d^\star R - 2R(Z))(v_2,v_3))\\
&+\tfrac{1}{2}(df_s)(\n_{v_1} (d^\star R -2R(Z))(v_2,v_3)) - \tr  (\n_{\cdot} df_s)((\n_{v_1}R)(\cdot,v_2)v_3)\\
&- \tr  ( \n_{v_1}\n_{\cdot} df_s)(R(\cdot,v_2)v_3).
\end{align}
By the Ricci identity, the final term on the right-hand side satisfies
\begin{align}
\tr  ( \n_{v_1}\n_{\cdot} df_s)(R(\cdot,v_2)v_3) =\,& \tr  ( \n_{\cdot} \n df_s)({v_1},R(\cdot,v_2)v_3)\\
& + \tr (df_s)(R(\cdot,v_1)R(\cdot,v_2)v_3).
\end{align}
Furthermore, by the relations given by Lemmas \ref{lem:commutator3} and \ref{lem:weitzenbock3}, we have
\begin{align}
&\tfrac{1}{2}(\n_{v_1} \square \n df_s)({v_2},{v_3})\\
=\,& \tfrac{1}{2}(\square \n \n df_s)({v_1},{v_2},{v_3}) - \tfrac{1}{2}(\n \n df_s) ({\Ric}^\sharp ({v_1}),{v_2},{v_3})\\
&  -\tfrac{1}{2}(\n df_s) ( (d^\star R)({v_1},{v_2}),{v_3}) - \tfrac{1}{2}(\n df_s) ({v_2},(d^\star R)({v_1},{v_3}))\\
& -\tr (\n_\cdot \n df_s)(R({v_1},\cdot){v_2},{v_3})- \tr (\n_\cdot \n df_s)({v_2},R({v_1},\cdot){v_3})
\end{align}
and
\begin{align}
(\n_{v_1} \n_Z \n df_s)({v_2},{v_3}) =\,& (\n_{Z}\n \n df_s)({v_1},{v_2},{v_3}) + (\n df_s)(R(Z)({v_1},{v_2}),{v_3})\\
&+ (\n df_s)({v_2},R(Z)({v_1},{v_3})) + (\n \n df_s)(\n_{v_1} Z,{v_2},{v_3})
\end{align}
so that
\begin{align}
&(\partial_s (\n \n df_s))(v_1,v_2,v_3) \\
=\,& -((\tfrac{1}{2}\square + \n_Z) \n \n df_s)({v_1},{v_2},{v_3})+ \tfrac{1}{2}(\n \n df_s) ({\Ric}_Z^\sharp ({v_1}),{v_2},{v_3})\\
&+\tfrac{1}{2}(\n\n df_s)(v_1,{\Ric}_Z^\sharp (v_2),v_3)+ \tfrac{1}{2}(\n \n df_s)(v_1,v_2,{\Ric}_Z^\sharp (v_3))\\
& +\tfrac{1}{2}(\n df_s) ( (d^\star R-2R(Z))({v_1},{v_2}),{v_3})\\
& + \tfrac{1}{2}(\n df_s) ({v_2},(d^\star R-2R(Z))({v_1},{v_3}))\\
& + \tfrac{1}{2}(\n df_s)(v_1,(d^\star R - 2R(Z))(v_2,v_3))\\
& +\tr (\n_\cdot \n df_s)(R({v_1},\cdot){v_2},{v_3})+ \tr (\n_\cdot \n df_s)({v_2},R({v_1},\cdot){v_3})\\
& - \tr  ( \n_{v_1}\n_{\cdot} df_s)(R(\cdot,v_2)v_3) - \tr (df_s)(R(\cdot,v_1)R(\cdot,v_2)v_3)\\
&- \tr  (\n_{\cdot} df_s)(\n_{v_1}R(\cdot,v_2)v_3)+ \tfrac{1}{2}(\n df_s)({v_2},(\n_{v_1} {\Ric}_Z^\sharp )(v_3))\\
& + \tfrac{1}{2}(\n df_s)({v_1},(\n_{v_2}{\Ric}_Z^\sharp)(v_3))+\tfrac{1}{2}(\n df_s)((\n_{v_1}{\Ric}_Z^\sharp) (v_2),v_3)\\
&+ \tfrac{1}{2}( df_s)((\n_{v_1}\n{\Ric}_Z^\sharp)(v_2,v_3))+\tfrac{1}{2}(df_s)(\n_{v_1} (d^\star R -2R(Z))(v_2,v_3))
\end{align}
at each $x \in M$ with $v_1,v_2,v_3 \in T_xM$. Now set
\begin{align}
W^3_s(u,v,w) :=\,& W_s(u) \otimes W_s(v) \otimes W_s(w) \\
W^2_s(u,v,w) :=\,& W'_s(u,v)\otimes W_s(w) + W_s(v)\otimes W'_s(u,w) + W_s(u)\otimes W'_s(v,w)\\
W^1_s(u,v,w) :=\,& W''_s(u,v,w)
\end{align}
and suppress the variables $(u,v,w)$ so that
\begin{equation}
N''_s= (\n \n df_s)(W^3_s) + (\n df_s)(W^2_s) + (df_s)(W^1_s).
\end{equation}
Then, by It\^{o}'s formula, we find
\begin{align}
dN''_s =\,& (\n_{//_s dB_s} \n \n df_s)(W^3_s) + (\n \n df_s)(DW^3_s)\\
& + (\partial_s + \tfrac{1}{2}\square + \n_Z)(\n \n df_s)(W^3_s)ds + [d(\n \n df),DW^3]_s\\
& + (\n_{//_s dB_s}\n df_s)(W^2_s) + (\n d f_s)(DW^2_s)ds\\
& + (\partial_s + \tfrac{1}{2}\square + \n_Z)(\n df_s)(W^2_s)ds + [d(\n df),DW^2]_s\\
& + (\n_{//_s dB_s} df_s)(W^1_s) + (d f_s)(DW^1_s)ds\\
& + (\partial_s + \tfrac{1}{2}\square + \n_Z)(df_s)(W^1_s)ds + [d(df),DW^1]_s.
\end{align}
Clearly 
\begin{align}
[d(\n \n df),DW^3]_s = 0
\end{align}
while furthermore
\begin{align}
[d(\n df),DW^2]_s =\,& [\nabla_{// dB} \nabla d f,R(//dB,W(u))W(v)\otimes W(w)]_s\\
&+ [\nabla_{// dB} \nabla d f,W(v)\otimes R(//dB,W(u))W(w)]_s\\
&+ [\nabla_{// dB} \nabla d f,W(u)\otimes R(//dB,W(v))W(w)]_s\\
=\,& -\tr (\nabla_{\cdot} \nabla d f_s)(R(W_s(u),\cdot)W_s(v),W_s(w))ds \\
& - \tr (\nabla_{\cdot} \nabla d f_s)(W_s(v),R(W_s(u),\cdot)W_s(w))ds \\
& + \tr (\nabla_{\cdot} \nabla d f_s)(W_s(u),R(\cdot,W_s(v))W_s(w))ds
\end{align}
and
\begin{align}
[d(df),DW^1]_s =\,& [\nabla_{// dB} d f, (\n_{W(u)} R)(//dB,W(v))W(w)]_s\\
& + [\nabla_{// dB} \nabla d f,R(//dB,W'(u,v))W(w) ]_s\\
& + [\nabla_{// dB} \nabla d f,R(//dB,W(v))W'(u,w) ]_s\\
& + [\nabla_{// dB} \nabla d f,R(//dB,W(u))W'(v,w) ]_s\\
=\,& \tr (\nabla_{\cdot} d f_s)( (\n_{W_s(u)} R)(\cdot,W_s(v))W_s(w))ds\\
& + \tr(\nabla_{\cdot} d f)(R(\cdot,W'_s(u,v))W_s(w))ds\\
& + \tr(\nabla_{\cdot} d f)(R(\cdot,W_s(v))W'_s(u,w))ds\\
& + \tr(\nabla_{\cdot} d f)(R(\cdot,W_s(u))W'_s(v,w))ds.
\end{align}
In calculating the second formula, we used
\begin{align}
& DW^2_s(u,v,w)\\
=\,& DW'_s(u,v) \otimes W_s(w) + W'_s(u,v) \otimes DW_s(w)\\
& +DW_s(v) \otimes W'_s(u,w) + W_s(v) \otimes DW'_s(u,w)\\
& +DW_s(u) \otimes W'_s(v,w) + W_s(u) \otimes DW'_s(v,w)\\
=\,& R(//_sdB_s,W_s( u ))W_s(v) \otimes W_s(w)\\
& -\tfrac{1}{2} (d^\star R -2 R(Z) +\nabla {\Ric}^\sharp_Z) (W_s( u),W_s(v)) \otimes W_s(w)ds\\
& -\tfrac{1}{2} {\Ric}^\sharp_Z (W^{\prime}_s(u,v))ds \otimes W_s(w) -\tfrac{1}{2} W'_s(u,v) \otimes {\Ric}^\sharp_Z (W_s(w))ds\\
& + W_s(v) \otimes R(//_sdB_s,W_s( u ))W_s(w)\\
& -\tfrac{1}{2} W_s(v) \otimes (d^\star R -2 R(Z) +\nabla {\Ric}^\sharp_Z) (W_s( u),W_s(w))ds\\
& -\tfrac{1}{2} W_s(v) \otimes  {\Ric}^\sharp_Z (W^{\prime}_s(u,w))ds -\tfrac{1}{2} {\Ric}^\sharp_Z (W_s(v)) ds \otimes W'_s(u,w)\\
& + W_s(u) \otimes R(//_sdB_s,W_s(v))W_s(w)\\
& -\tfrac{1}{2} W_s(u) \otimes  (d^\star R -2 R(Z) +\nabla {\Ric}^\sharp_Z) (W_s(v),W_s(w))ds\\
& -\tfrac{1}{2} W_s(u) \otimes  {\Ric}^\sharp_Z (W^{\prime}_s(v,w))ds -\tfrac{1}{2} {\Ric}^\sharp_Z (W_s(u)) ds \otimes W'_s(v,w)
\end{align}
while the analogous formulas for $DW^1_s$ and $DW^3_s$ follow directly from the definitions. Putting all this together we see that $dN''_s$ is given by an expression involving $74$ terms, $64$ of which cancel, leaving
\begin{align}
dN''_s(u,v,w) =\,& (\n_{//_s dB_s} \n \n df_s)(W^3_s(u,v,w))\\
& + (\n_{//_s dB_s}\n df_s)(W^2_s(u,v,w))\\
& + (\n d f_s)(R(//_sdB_s,W_s( u ))W_s(v),W_s(w))\\
& + (\n d f_s)(W_s(v),R(//_sdB_s,W_s( u ))W_s(w))\\
& + (\n d f_s)(W_s(u),R(//_sdB_s,W_s(v))W_s(w))\\
& + (\n_{//_s dB_s} df_s)(W^1_s(u,v,w))\\
& + (d f_s)((\n_{W_s(u)} R)(//_sdB_s,W_s(v))W_s(w))\\
& + (d f_s)(R(//_sdB_s,W'_s(u,v))W_s(w))\\
& + (d f_s)R(//_sdB_s,W_s(v))W'_s(u,w))\\
& + (d f_s)R(//_sdB_s,W_s(u))W'_s(v,w))
\end{align}
so that $N''_s(u,v,w)$ is, in particular, a local martingale.
\end{proof}

\begin{theorem}\label{thm:thirdderform}
If the local martingale $N''_s(u,v,w)$ is a martingale with $f \in C^3_b(M)$ then
\begin{align}
(\n \n d P_t f)(u,v,w) =\,& \E[(\n \n df)(W_t(u),W_t(v),W_t(w))]\\
& + \E[(\n df)(W_t(v),W'_t(u,w))]\\
& + \E[(\n df)(W'_t(u,v),W_t(w))]\\
& + \E[(\n df)(W_t(u),W'_t(v,w))]\\
& + \E[(df)(W''_t(u,v,w))]
\end{align}
for all $u,v \in T_xM$, $x \in M$ and $t \geq 0$.
\end{theorem}

\begin{proof}
This follows from Lemma \ref{lem:N''locmart} since if $N''_s(u,v,w)$ is a martingale then
\begin{equation}
\E[N''_0(u,v,w)] = \E[N''_t(u,v,w)]
\end{equation}
for all $t \geq 0$.
\end{proof}

To reduce the number of derivatives appropriately, we must now perform integration by parts, at the level of local martingales.

\begin{lemma}\label{lem:ibp21}
Suppose $k$ is a bounded adapted process with paths in the Cameron-Martin space $L^{1,2}([0,t];\Aut(T_{x}M))$. Suppose $k_s = 0$ for $s \geq \tau \wedge t$. Then
\begin{align}
&N''_s(k_su,v,w) - (df_s)(W_s(w))\int_0^s \<W'_r(\dot{k}_r u,v),//_rdB_r\>\\
& - (df_s)(W_s(v))\int_0^s \<W'_r(\dot{k}_ru,w),//_r dB_r\>\\
& - ((\n d f_s)(W_s(v),W_s(w)) + (df_s)(W_s'(v,w)))\int_0^s \< W_r(\dot{k}_r u),//_r dB_r\>\\
& + \int_0^s (df_r)(R(W_r(\dot{k}_ru),W_r(v))W_r(w)-W''_r(\dot{k}_r u,v,w))dr
\end{align}
is a local martingale, for each $u,v,w \in T_xM$.
\end{lemma}

\begin{proof}
Since according to Lemma \ref{lem:N''locmart} $N''_s(u,v,w)$ is a local martingale, it follows that
\begin{equation}
N''_s(k_su,v,w) - \int_0^s N''_r(\dot{k}_r u,v,w) dr
\end{equation}
is a local martingale. We saw in the proof of Lemma \ref{lem:N'locmart} that
\begin{align}
dN'_s(v,w) =\,& (\nabla_{//_s dB_s} \nabla d f_s)(W_s(v),W_s(w)) + (\nabla_{//_sdB_s} df_s)(W^{\prime}_s(v,w))\\
& + (df_s)(R(//_s dB_s,W_s(v))W_s(w))
\end{align}
and so, by integration by parts, it follows that
\begin{align}
&\int_0^s (\n \n df_r)(W_r(\dot{k}_r u),W_r(v),W_r(w))dr\\
& - ((\n d f_s)(W_s(v),W_s(w)) + (df_s)(W_s'(v,w)))\int_0^s \< W_r(\dot{k}_r u),//_r dB_s\>\\
& + \int_0^s (\n df_r)(W_r(\dot{k}_s u), W'_r(v,w))dr + \int_0^s (df_r)(R(W_r(\dot{k}_ru),W_r(v))W_r(w))dr
\end{align}
is a local martingale and therefore
\begin{align}
&N''_s(k_tu,v,w) - \int_0^s (\n df_r)(W_r(v),W_r'(\dot{k}_r u,w))ds\\
& - \int_0^s (\n df_r)(W_r'(\dot{k}_r u,v),W_r(w))ds\\
& - ((\n d f_s)(W_s(v),W_s(w)) + (df_s)(W_s'(v,w)))\int_0^s \< W_r(\dot{k}_r u),//_r dB_s\>\\
& + \int_0^s (df_r)(R(W_r(\dot{k}_ru),W_r(v))W_r(w))dr\\
& - \int_0^s (df_r)(W''_r(\dot{k}_r u,v,w))dr
\end{align}
is a local martingale. Furthermore, by Lemma \ref{lem:Nlocmart} and integration by parts, we see that
\begin{equation}
\int_0^s (\n df_r)(\n df_r)(W_r(v),W_r'(\dot{k}_r u,w))ds - (df_s)(W_s(v))\int_0^s \<W'_r(\dot{k}_ru,w),//_r dB_r\>
\end{equation}
and similarly
\begin{equation}
\int_0^s (\n df_r)(W'_r(\dot{k}_ru,v),W_r(w))ds - (df_s)(W_s(w))\int_0^s \<W'_r(\dot{k}_r u,v),//_rdB_r\> 
\end{equation}
are local martingales. For this we used the fact that the Hessian is symmetric. Putting all this together completes the proof of the lemma.
\end{proof}

\begin{theorem}\label{thm:secondderhess}
If the local martingale appearing in Lemma \ref{lem:ibp21} is a martingale with $f \in C^2_b(M)$ then
\begin{align}
& (\n \n d P_t f)(u,v,w)\\
=\,& -\E\left[(df)(W_t(w))\int_0^t \<W'_s(\dot{k}_s u,v),//_sdB_s\>\right]\\
& - \E\left[(df)(W_t(v))\int_0^t \<W'_s(\dot{k}_su,w),//_s dB_s\>\right]\\
& - \E\left[((\n d f)(W_t(v),W_t(w)) + (df)(W_t'(v,w)))\int_0^t \< W_s(\dot{k}_s u),//_s dB_s\>\right]\\
&+ \E\left[(df)(W_t\int_0^t W_s^{-1}(R(W_s(\dot{k}_su),W_s(v))W_s(w)-W''_s(\dot{k}_s u,v,w))ds)\right]
\end{align}
for all $u,v \in T_xM$, $x \in M$ and $t \geq 0$.
\end{theorem}

\begin{proof}
This follows from Lemma \ref{lem:ibp21} by taking expectations and evaluating at times $s=0$ and $s=t$.
\end{proof}

We now perform a second integration by parts, which again we do at the level of local martingales. For this we suppose $D_1$ and $D_2$ are regular domains with $x \in D_1$ and $D_1$ compactly contained in $D_2$. We denote by $\tau_1$ and $\tau_2$ the first exit times of $X(x)$ from $D_1$ and $D_2$, respectively.

\begin{lemma}\label{lem:ibp2}
Suppose $k,l$ are bounded adapted processes with paths in the Cameron-Martin space $L^{1,2}([0,t];\Aut(T_{x}M))$. Fix $0<t_1<t$. Suppose $k_s = 0$ for $s \geq \tau_1 \wedge t_1$ with $l_s = \id_{T_xM}$ for $0 \leq s \leq \tau_1 \wedge t_1$ and $l_s = 0$ for $s \geq \tau_2 \wedge t$. Then
\begin{align}
&N''_s(k_su,v,w) - (df_s)(W_s(w))\int_0^s \<W'_r(\dot{k}_r u,v),//_rdB_r\>\\
& - (df_s)(W_s(v))\int_0^s \<W'_r(\dot{k}_ru,w),//_r dB_r\>\\
& - ((\n d f_s)(W_s(l_s v),W_s(w)) + (df_s)(W_s'(l_s v,w)))\int_0^s \< W_r(\dot{k}_r u),//_r dB_s\>\\
& + \int_0^s (df_r)(R(W_r(\dot{k}_ru),W_r(v))W_r(w)-W''_r(\dot{k}_r u,v,w))dr\\
& + (df_s)(W_s(w))\int_0^s \langle W_r(\dot{l}_rv),//_rdB_r\rangle\int_0^s \< W_r(\dot{k}_r u),//_rdB_r\>\\
& + \int_0^s  (df_r)(W^{\prime}_r (\dot{l}_r v,w))dr\int_0^s \< W_r(\dot{k}_r u),//_rdB_r\>
\end{align}
is a local martingale, for each $u,v,w \in T_xM$.
\end{lemma}

\begin{proof}
By Lemma \ref{lem:ibp1}, it follows that
\begin{align}
O_s^{(1)}:=\,& (\nabla d f_s)(W_s((l_s-1) v),W_s(w)) + (df_s)(W^{\prime}_s ((l_s-1) v,w)) \\
&- (df_s)(W_s(w))\int_0^s \langle W_r(\dot{l}_rv),//_rdB_r\rangle - \int_0^s  (df_r)(W^{\prime}_r (\dot{l}_r v,w))dr
\end{align}
is a local martingale and so is
\begin{equation}
O_s^{(2)}:= \int_0^s \< W_r(\dot{k}_r u),//_rdB_r\>
\end{equation}
and therefore so is the product $O_s^{(1)}O_s^{(2)}$ since $O_s^{(1)} = 0$ on $[0,\tau_1 \wedge t_1]$ with $O_s^{(2)}$ constant on $[\tau_1 \wedge t_1, \tau_2 \wedge t]$. Subtracting the local martingale $O_s^{(1)}O_s^{(2)}$ from the local martingale appearing in Lemma \ref{lem:ibp21} therefore completes the proof of the lemma.
\end{proof}

\begin{theorem}\label{thm:locformthree}
Suppose ${\Ric}_Z$ is bounded below with $f \in C^1_b(M)$. Fix $t>0$, suppose $k,l$ are as in Lemma \ref{lem:ibp2} with
\begin{equation}
\E\left[\int_0^t |\dot{k}_s|^2 ds\right]< \infty,\quad \E\left[\int_0^t |\dot{k}_s|^2 ds\right] < \infty.
\end{equation}
Then
\begin{align}
&(\n \n dP_tf)(u,v,w)\\
=\,& -\E\left[ (df)(W_t(w))\int_0^t \<W'_s(\dot{k}_s u,v),//_sdB_s\>\right]\\
& -\E\left[ (df)(W_t(v))\int_0^t \<W'_s(\dot{k}_s u,w),//_sdB_s\>\right]\\
& + \E\left[(df)(W_t(w))\int_0^t \langle W_s(\dot{l}_sv),//_sdB_s\rangle\int_0^t \< W_s(\dot{k}_s u),//_sdB_s\>\right]\\
& + \E\left[(df)(W_t\int_0^t W_s^{-1}(R(W_s(\dot{k}_su),W_s(v))W_s(w)-W''_s(\dot{k}_s u,v,w))ds)\right]\\
& + \E\left[(df)(W_t\int_0^t  W_s^{-1}(W^{\prime}_s (\dot{l}_s v,w))ds\int_0^t \< W_s(\dot{k}_s u),//_sdB_s\>)\right]
\end{align}
for all $u,v,w \in T_xM$.
\end{theorem}

\begin{proof}
Since the domains are compact it follows that the local martingale appearing in Lemma \ref{lem:ibp2} a true martingale on $[0,t]$ and consequently
\begin{align}
&(\n \n dP_tf)(u,v,w)\\
=\,& -\E\left[ (df)(W_t(w))\int_0^t \<W'_s(\dot{k}_s u,v),//_sdB_s\>\right]\\
& -\E\left[ (df)(W_t(v))\int_0^t \<W'_s(\dot{k}_s u,w),//_sdB_s\>\right]\\
& + \E\left[(df_t)(W_t(w))\int_0^t \langle W_s(\dot{l}_sv),//_sdB_s\rangle\int_0^t \< W_s(\dot{k}_s u),//_sdB_s\>\right]\\
& + \E\left[\int_0^t (df_s)(R(W_s(\dot{k}_su),W_s(v))W_s(w)-W''_s(\dot{k}_s u,v,w))ds\right]\\
& + \E\left[\int_0^t  (df_s)(W^{\prime}_s (\dot{l}_s v,w))ds\int_0^t \< W_s(\dot{k}_s u),//_sdB_s\>\right]
\end{align}
Now apply Theorem \ref{thm:locformone} and the Markov property to the final two terms on the right-hand side.
\end{proof}

In the next section, we will use Theorems \ref{thm:locformone}, \ref{thm:locformtwo} and \ref{thm:locformthree} to deduce appropriate estimates on the derivatives of $P_tf$.

\section{Derivative estimates}\label{sec:derests}

\begin{theorem}\label{thm:estone}
Suppose $K \in \R$ with $\Ric_Z \geq 2K$ and $f \in C^1_b(M)$. Then
\begin{equation}
\|\n P_tf\|_\infty \leq e^{-Kt} \|\n f\|_\infty
\end{equation}
for all $t \geq 0$.
\end{theorem}

\begin{proof}
This estimate, contained in Theorem \ref{thm:exponentialconv}, follows directly from Theorem \ref{thm:locformone}. In particular, if ${\Ric}_Z \geq 2K$ then Gronwall's inequality implies $|W_t| \leq e^{-K t}$ for all $t \geq 0$.
\end{proof}

\begin{theorem}\label{thm:esttwo}
Suppose $K \in \R$ with $\Ric_Z \geq 2K$ and $f \in C^1_b(M)$. Suppose
\begin{equation}
\|R\|_\infty < \infty ,\quad \|\nabla {\Ric}^\sharp_Z + d^\star R -2R(Z)\|_\infty < \infty.
\end{equation}
Then there exists a positive constant $C_1$ such that
\begin{equation}
\|\n d P_tf\|_\infty \leq  \frac{C_1 e^{-Kt}}{\sqrt{1\wedge t}} \|\n f\|_\infty 
\end{equation}
for all $t > 0$.
\end{theorem}

\begin{proof}
Take an exhaustion $\{D_i\}_{i=1}^\infty$ of $M$ by regular domains. For $x \in D_i$ with $k_s^i$ a real-valued Cameron-Martin process satisfying the conditions of Theorem \ref{thm:locformtwo}, we have
\begin{align}
&(\n dP_tf)(u,v) = \\
& -\E\left[ (df)(W_t(v))\int_0^t \langle W_s(\dot{k}^i_s u),//_sdB_s\rangle\right]\\
&- \E\left[ (df)(W_t\int_0^t\dot{k}^i_s\int_0^s W_r^{-1} R(//_r dB_r,W_r(u))W_r (v)ds)\right]\\
&+ \frac{1}{2}\E\left[ (df)(W_t\int_0^t \dot{k}^i_s \int_0^s W_r^{-1}(\nabla {\Ric}^\sharp_Z + d^\star R -2R(Z))(W_r ( u),W_r(v))dr ds) \right].
\end{align}
Since we assume the curvature operators are uniformly bounded, it follows by taking limits that the above formula holds with $k_s^i$ replaced by $k_s$ where
\begin{equation}
k_s := \frac{(1 \wedge t) - s}{1 \wedge t} \vee 0,\quad \dot{k}_s =
\begin{cases}
-1/(1 \wedge t), & 0 \leq s < 1\wedge t\\
0, & s \geq 1\wedge t
\end{cases}.
\end{equation}
With this choice of $k_s$, the desired estimate follows from the above equation by the Cauchy-Schwarz inequality and It\^{o} isometry.
\end{proof}

\begin{theorem}\label{thm:estthree}
Suppose $K \in \R$ with $\Ric_Z \geq 2K$. Suppose
\begin{equation}
\begin{aligned}
&\|R\|_\infty < \infty,\\
&\|\n R\|_\infty < \infty,
\end{aligned}\quad
\begin{aligned}
&\|\nabla {\Ric}^\sharp_Z + d^\star R -2R(Z)\|_\infty < \infty,\\
&\|\n(\nabla {\Ric}^\sharp_Z + d^\star R -2R(Z))\|_\infty < \infty.\\
\end{aligned}
\end{equation}
If $f \in C^1_b(M)$ then there exists a positive constant $C_2$ such that
\begin{equation}
\|\n\n d P_tf\|_\infty \leq  \frac{C_2 e^{-Kt}}{1\wedge t} \|\n f\|_\infty
\end{equation}
for all $t > 0$. If $f \in C^2_b(M)$ then there exists a positive constant $C_3$ such that
\begin{equation}
\|\n\n d P_tf\|_\infty \leq  \frac{C_3 e^{-Kt}}{\sqrt{1\wedge t}} (\|\n f\|_\infty + \|\Hess f\|_\infty)
\end{equation}
for all $t > 0$.
\end{theorem}

\begin{proof}
With the Cameron-Martin process $k_s$ chosen as in the proof of Theorem \ref{thm:esttwo}, these estimates follow from Theorems \ref{thm:locformthree} and \ref{thm:secondderhess} in a similar manner, simply by estimating each term in the formula using the Cauchy-Schwarz inequality and It\^{o} isometry.
\end{proof}

The three theorems of this section provide precisely the estimates required in the proofs of the main results in Section \ref{sec:6}.

\end{document}